\documentclass[12pt]{article}
\usepackage{amsmath, amssymb,latexsym}
\usepackage{color}
\title{From semigroups to subelliptic estimates for quadratic operators}
\author{Michael Hitrik \\\small Department of Mathematics \\\small University of California
\\\small Los Angeles \\\small CA 90095-1555, USA\\\small hitrik@math.ucla.edu \and
Karel Pravda-Starov\\\small IRMAR, CNRS UMR 6625 \\\small Universit\'e de Rennes 1\\\small Campus de Beaulieu\\\small 263 avenue du G\'en\'eral Leclerc, CS 74205
\\\small 35042 Rennes Cedex, France\\\small karel.pravda-starov@univ-rennes1.fr \and
Joe Viola \\\small Laboratoire de Math\'ematiques Jean Leray\\\small Universit\'e de Nantes\\\small 2, rue de la Houssini\`ere \\\small BP 92208,
44322 Nantes Cedex 3, France\\\small Joseph.Viola@univ-nantes.fr}
\date{}

\def\wrtext#1{\relax\ifmmode{\leavevmode\hbox{#1}}\else{#1}\fi}
\def\abs#1{\left|#1\right|}
\def\begeq{\begin{equation}}
\def\endeq{\end{equation}}

\newcommand{\eps}{\varepsilon}
\def\part#1{\frac{\partial}{\partial #1}}

\def\norm#1{||\,#1\,||}

\newcommand{\real}{\mbox{\bf R}}
\newcommand{\comp}{\mbox{\bf C}}

\newcommand{\nat}{\mbox{\bf N}}

\renewcommand{\exp}{\mbox{\rm exp\,}}

\newtheorem{dref}{Definition}[section]

\newtheorem{theo}[dref]{Theorem}
\newtheorem{prop}[dref]{Proposition}

\newenvironment{proof}{\vspace{.3cm}\noindent{{\em Proof:}}}{\hfill$\Box$}

\begin{document}

\maketitle

\vspace*{1cm}
\noindent
{\bf Abstract}: Using an approach based on the techniques of FBI transforms, we give a new simple proof of the global subelliptic estimates for
non-selfadjoint non-elliptic quadratic differential operators, under a natural averaging condition on the Weyl symbols of the operators, established by
the second author~\cite{PS}. The loss of the derivatives in the subelliptic estimates depends directly on algebraic properties of the Hamilton maps of the quadratic
symbols. Using the FBI point of view, we also give accurate smoothing estimates of Gelfand-Shilov type for the associated heat semigroup in the limit of small
times.



%

\tableofcontents
\section{Introduction and statement of results}
\setcounter{equation}{0}
There has recently been a large number of new developments for non-selfadjoint quadratic differential operators. Here some of the motivation comes
from the analysis of second-order operators of Kramers-Fokker-Planck type, where non-selfadjoint non-elliptic quadratic operators often arise as local
models via harmonic oscillator approximation,~\cite{HelfferNier},~\cite{HeSjSt},~\cite{HPS13},~\cite{Nier}. The recent results in the quadratic case include
precise bounds on the resolvent and spectral projections,~\cite{HiSjVi},~\cite{Viola}, determination of spectra for partially elliptic operators~\cite{HPS09},
as well as smoothing and decay estimates for the corresponding semigroup in the limit of large times,~\cite{HPS09},~\cite{OPaPS},~\cite{AlVi},~\cite{GM}. Of
particular relevance for the present paper is the work~\cite{PS} by the second author, where global subelliptic estimates are established for the class of non-selfadjoint non-elliptic quadratic
differential operators, whose Weyl symbols enjoy certain dynamical averaging properties, studied in~\cite{HPS09}. The purpose of this work is to develop a
new time-dependent approach to the results of~\cite{PS}. When doing so, we shall also establish a precise form of the Gelfand-Shilov regularizing property
for the associated quadratic semigroup, in the small time limit. Let us proceed now to describe the assumptions, state the results, and outline the main
ideas of the proofs.

\bigskip
\noindent
Let $q$ be a complex valued quadratic form on the phase space,
\begeq
\label{eq1.1}
q: \real^n_x\times \real^n_{\xi} \to \comp,\quad (x,\xi)\mapsto q(x,\xi).
\endeq
We shall assume throughout the following discussion that the quadratic form ${\rm Re}\, q$ is positive semi-definite,
\begeq
\label{eq1.41}
{\rm Re}\,q(x,\xi) \geq 0,\quad (x,\xi)\in \real^{2n}.
\endeq
Associated to $q$ is the Hamilton map,
\begeq
\label{eq1.42}
F: \comp^{2n}\rightarrow \comp^{2n},
\endeq
defined by the identity
\begeq
\label{eq1.4}
q(X,Y) = \sigma(X,FY),\quad X,Y\in \comp^{2n}.
\endeq
Here $\sigma$ is the complex symplectic form on $\comp^{2n}$ and the left hand side is the polarization of $q$, viewed as a symmetric bilinear form on
$\comp^{2n}$. We notice that the Hamilton map $F$ is skew-symmetric with respect to $\sigma$.

\bigskip
\noindent
Following~\cite{HPS09},~\cite{PS}, let us introduce the so called singular space associated to $q$,
\begeq
\label{eq1.5}
S = \Big(\bigcap_{j=0}^{2n-1}\textrm{Ker}\big[\textrm{Re }F(\textrm{Im }F)^j \big]\Big) \bigcap \real^{2n}.
\endeq
As established in \cite[Section 2]{HPS09}, an equivalent description of the linear subspace $S \subset \real^{2n}$ can be given as follows,
\begeq
\label{eq1.6}
S= \Big\{X \in \real^{2n};\,\,H_{\textrm{Im}\, q}^k\textrm{Re }q(X)=0,\,\, k \in \nat\Big\}.
\endeq
Here $H_f = f'_{\xi}\cdot \partial_x - f'_x\cdot \partial_{\xi}$ is the Hamilton vector field of a function $f\in C^1(\real^{2n}_{x,\xi};\real)$.

\medskip
\noindent
Throughout this work, it will be assumed that the singular space of $q$ is trivial,
\begeq
\label{eq1.51}
S = \{0\}.
\endeq
It was shown in~\cite[Proposition 2.0.1]{HPS09} that the assumption (\ref{eq1.51}) implies that for each $T>0$, the quadratic form
$$
\real^{2n} \ni X \mapsto \int_0^T {\rm Re}\, q\left(\exp(tH_{{\rm Im}\,q})(X)\right)\,dt
$$
is positive definite.

\bigskip
\noindent
In view of (\ref{eq1.51}), we may introduce $0\leq k_0 \leq 2n-1$ to be the smallest integer such that
\begeq
\label{eq1.7}
\Big(\bigcap_{j=0}^{k_0}\textrm{Ker}\big[\textrm{Re }F(\textrm{Im }F)^j \big]\Big) \bigcap \real^{2n} = \{0\}.
\endeq
Let us recall from~\cite[Section 2]{HPS09},~\cite{HPS13} that (\ref{eq1.7}) implies the following subelliptic condition for the quadratic symbol $q$:
for each $0\neq X\in \real^{2n}$, there exists an integer $j\in \{0,\ldots, k_0\}$ such that
\begeq
\label{eq1.71}
{\rm Re}\,q\left(\exp(tH_{{\rm Im}\,q})(X)\right) = at^{2j} + {\cal O}(t^{2j+1}),\quad t\rightarrow 0,
\endeq
where $a = (H_{{\rm Im}\,q}^{2j}{\rm Re}\,q)(X)/(2j)! > 0$.

\bigskip
\noindent
The following result was established in~\cite{PS}, where we write
\begeq
q^w(x,D_x) u(x) = \frac{1}{(2\pi)^n}\int\!\!\!\int e^{i(x-y)\cdot\xi} q\left(\frac{x+y}{2},\xi\right)u(y)\,dy\,d\xi,\quad u\in {\cal S}(\real^n),
\endeq
for the Weyl quantization of $q$. Before stating it, let us introduce the unbounded selfadjoint operator
$$
\langle{(x,D_x)\rangle}^r = \left(1+x^2¨+D_x^2\right)^{r/2},\quad r > 0,
$$
defined in terms of the functional calculus for the selfadjoint operator $D_x^2+x^2$ on $L^2(\real^n)$. Here $D_x = i^{-1}\partial_x$.

\begin{theo}
\label{theo_main}
Let $q: \real^n_x\times \real^n_{\xi} \to \comp$ be a quadratic form with ${\rm Re}\,q\geq 0$, such that {\rm (\ref{eq1.51})} holds. Let
$k_0\in \{0,\ldots, 2n-1\}$ be the smallest integer such that
$$
\Big(\bigcap_{j=0}^{k_0}{\rm Ker}\,\big[{\rm Re}\,F({\rm Im}\,F)^j \big]\Big) \bigcap \real^{2n} = \{0\}.
$$
Then there exists $C>0$ such that for all $u\in {\cal S}(\real^n)$, we have
\begeq
\label{eq1.8}
\norm{\langle{(x,D_x)\rangle}^{2/(2k_0+1)}u}_{L^2({\bf R}^n)} \leq C\left(\norm{q^w(x,D_x)u}_{L^2({\bf R}^n)} + \norm{u}_{L^2({\bf R}^n)}\right).
\endeq
\end{theo}

\medskip
\noindent
The global subelliptic estimate (\ref{eq1.8}) has turned out to be crucial in~\cite{OPaPS}, in particular, when deriving sharp resolvent estimates for
$q^w(x,D_x)$ in suitable parabolic neighborhoods of the imaginary axis and when showing the exponential rate of convergence to equilibrium for the associated
semigroup. The proof of Theorem 1.1 given in~\cite{PS} is quite technical and is based on a delicate construction of a bounded real multiplier $G \in C^{\infty}(\real^{2n})\cap L^{\infty}(\real^{2n})$
such that
$$
{\rm Re}\, q(X) + \frac{1}{C}H_{{\rm Im}\,q}G(X) +1 \geq \frac{1}{C} \langle{X\rangle}^{\frac{2}{2k_0+1}},\quad X\in \real^{2n}.
$$
Here $\langle{X\rangle} = \left(1 + \abs{X}^2\right)^{1/2}$. We remark that the existence of the multiplier $G$ has been a
key ingredient in the proof of subelliptic estimates for some classes of non-selfadjoint $h$-pseudodifferential operators with double characteristics
in~\cite{HPS13}.

\medskip
\noindent
In this paper, we shall give a new and simple proof of Theorem \ref{theo_main}, based on the study of the heat semigroup generated by the
quadratic differential operator $q^w = q^w(x,D_x)$, viewed as a Fourier integral operator with a quadratic complex phase. Indeed, rather than studying
the semigroup directly, following~\cite{HeSjSt},~\cite{HPS09},~\cite{Sj10}, we first perform a metaplectic FBI transform and consider the semigroup generated by
the holomorphic quadratic operator $Q$, obtained by conjugating $q^w(x,D_x)$ by the inverse of the FBI transform. We may then view the operator $e^{-tQ}$
as a quadratic Fourier integral operator in the complex domain~\cite{Sj82},~\cite{CaGrHiSj}, which is bounded between
weighted spaces of holomorphic functions,
$$
e^{-tQ}: H_{\Phi_0}(\comp^n) \rightarrow H_{\Phi_t}(\comp^n), \quad 0 \leq t \leq t_0 \ll 1.
$$
Here $\Phi_0$ is a strictly plurisubharmonic quadratic form on $\comp^n$ and
$$
H_{\Phi_0}(\comp^n) = {\rm Hol}(\comp^n)\cap L^2(\comp^n; e^{-2\Phi_0}\, L(dx))
$$
is the image of $L^2(\real^n)$ under the FBI transform, with $L(dx)$ standing for the Lebesgue measure on $\comp^n$. The space $H_{\Phi_t}(\comp^n)$ is defined
similarly, for the quadratic form $\Phi_t \leq \Phi_0$, whose time evolution is governed by a real Hamilton-Jacobi equation --- see the discussion in Section 2.
In~\cite{HPS09} we established, as a consequence of (\ref{eq1.51}), that for all $t>0$ small enough, we have the strict inequality
$\Phi_t < \Phi_0$ on $\comp^n\setminus\{0\}$. Here, crucially, we are able to sharpen this bound and to show that the assumption (\ref{eq1.51}) actually
implies that
\begeq
\label{eq1.81}
\Phi_t(x) \leq \Phi_0(x) - \frac{t^{2k_0 +1}}{C}\abs{x}^2,\quad x\in \comp^n,
\endeq
for all $t\geq 0$ small enough. The estimate (\ref{eq1.81}) is established here very much following the techniques of~\cite{Sj10}, developed
when studying subelliptic resolvent estimates for non-selfadjoint $h$-pseudodifferential operators of principal type. In particular, rather than working directly
with the Hamilton-Jacobi equation for $\Phi_t$, as was done in~\cite{HPS09}, following~\cite{Sj10}, we apply the inverse of the canonical transformation associated to the FBI transform, to
replace the Hamilton-Jacobi equation for $\Phi_t$ by another, closely related, one, which becomes easier to handle. Let us also mention that for the quadratic
Kramers-Fokker-Planck operator, where we have $k_0 = 1$, the estimate (\ref{eq1.81}) has been proved in~\cite[Section 11]{HeSjSt}.

\medskip
\noindent
With the estimate (\ref{eq1.81}) available, the proof of Theorem \ref{theo_main} may be completed by writing
\begeq
\label{eq1.82}
(Q-z)^{-1} = \int_0^{\infty} e^{-t(Q-z)}\,dt,\quad {\rm Re}\, z < 0,
\endeq
and observing that the smoothing properties for the resolvent $(Q-z)^{-1}$, equivalent to (\ref{eq1.8}), may be derived from (\ref{eq1.81}), via (\ref{eq1.82}),
essentially by carrying out the $t$-integration.

\bigskip
\noindent
It turns out that the new method of proof of Theorem \ref{theo_main} leads also to some accurate smoothing estimates for the contraction semigroup $e^{-tq^w}$,
$t\geq 0$, in the limit of small times. Specifically, let us recall from~\cite[Theorem 1.2.1]{HPS09} that under the assumptions (\ref{eq1.41}) and (\ref{eq1.51}), we have for any $t > 0$,
$$
e^{-tq^w}: L^2(\real^n) \rightarrow {\cal S}(\real^n).
$$
Here ${\cal S}(\real^n)$ is the Schwartz space. The behavior of the Schwartz seminorms of $e^{-tq^w}u$, for $u\in L^2(\real^n)$, as $t\rightarrow 0^+$, is
given in the following result.

\begin{theo}
\label{theo_main2}
Let $q: \real^n_x\times \real^n_{\xi} \to \comp$ be a quadratic form with ${\rm Re}\,q\geq 0$, such that {\rm (\ref{eq1.51})} holds and let us define the
integer $k_0\in \{0,\ldots\,2n-1\}$ as in Theorem {\rm {\ref{theo_main}}}. There exist $C > 0$ and $t_0 > 0$ such that for all $t\in (0,t_0]$, and all
$N\in \nat$, we have
\begeq
\label{eq1.9}
\norm{(D_x^2 + x^2)^N e^{-tq^w}}_{{\cal L}(L^2({\bf R}^n), L^2({\bf R}^n))} \leq \frac{C^{N+1}N!}{t^{(2k_0+1)N}}.
\endeq
Furthermore, there exist $C>0$ and $t_0 > 0$ such that for all $u\in L^2(\real^n)$, all $\mu, \nu \in \nat^n$, and all $t\in (0,t_0]$,
we have
\begeq
\label{eq1.10}
\norm{x^{\mu}\partial_x^{\nu}e^{-tq^w}u}_{L^{\infty}({\bf R}^n)} \leq
\frac{C^{1 + \abs{\mu} + \abs{\nu}} \left(\mu!\right)^{1/2} \left(\nu!\right)^{1/2}}{t^{\frac{(2k_0+1)}{2}(\abs{\mu} + \abs{\nu}+2n+s)}}\norm{u}_{L^2({\bf R}^n)}.
\endeq
Here $s > n/2$ is a fixed integer.
\end{theo}

\medskip
\noindent
Theorem \ref{theo_main2} implies that for any $t>0$ and $u\in L^2(\real^n)$, the function $e^{-tq^w}u$ belongs to the Gelfand-Shilov space
$S^{1/2}_{1/2}(\real^n)$, with the precise control of the Gelfand-Shilov seminorms, as $t\rightarrow 0^+$, described by the bounds (\ref{eq1.10}).
Here, following~\cite{LMPSXu}, we recall that a function $f\in C^{\infty}(\real^n)$ belongs to the Gelfand-Shilov
space $S^p_q(\real^n)$, with $p$, $q>0$, $p+q \geq 1$, if there exists a constant $C\geq 1$ such that for all $\mu,\nu \in \nat^n$, we have
$$
\norm{x^{\mu}\partial_x^{\nu}f}_{L^{\infty}({\bf R}^n)} \leq C^{1 + \abs{\mu} + \abs{\nu}} \left(\mu!\right)^{q} \left(\nu!\right)^{p}.
$$
We refer to~\cite{LMPSXu} and the references given there for a detailed discussion of the Gelfand-Shilov regularity theory. In the work~\cite{HPSV1},
prepared simultaneously with the present one, using direct methods, we carry out a more detailed study of the smoothing properties of the semigroup
$e^{-tq^w}$ in the small time limit, depending on phase space directions. Comparing Theorem \ref{theo_main2} with Corollary 1.2 in~\cite{HPSV1}, we
observe that the former result is sharper, since it provides an ${\cal O}(t^{-\frac{(2k_0+1)}{2}(\abs{\mu} + \abs{\nu} + 2n +s)})$ control for
the Gelfand-Shilov seminorms of $e^{-tq^w}u$ in the space $S^{1/2}_{1/2}(\real^n)$, whereas Corollary 1.2 in~\cite{HPSV1} gives a control in
$t^{-\frac{(2k_0+1)}{2}(\abs{\mu} + \abs{\nu} + s)}$ for the Gelfand-Shilov seminorms in the space $S^{\frac{2k_0 + 1}{2}}_{\frac{2k_0+1}{2}}(\real^n)$.

\medskip
\noindent
\color{black}
The plan of the paper is as follows. In Section 2, we study the semigroup generated by $q^w$ on the FBI transform side and establish the
estimate (\ref{eq1.81}). Representing the resolvent of $q^w$ as the Laplace transform of the semigroup and making use of some further direct arguments on the FBI
transform side, we complete the proof of Theorem \ref{theo_main} in Section 3. In Section 4, we establish Theorem \ref{theo_main2}, relying upon (\ref{eq1.81}),
essentially by comparing the semigroup generated by $q^w$ with that of the harmonic oscillator.

\bigskip
\noindent
{\bf Acknowledgements}. We are very grateful to Johannes Sj\"ostrand for helpful discussions. The first author would like to thank the Centre Henri Lebesgue at
the University of Rennes 1 for the kind hospitality in June 2015, where part of this project was conducted. The research of
the second and third authors is supported by the ANR NOSEVOL (Project: ANR 2011 BS0101901).

\section{The heat semigroup on the FBI transform side}\label{all}
\setcounter{equation}{0}
We shall view $q^w(x,D_x)$ as a closed densely defined operator on $L^2(\real^n)$, equipped with the domain
\begeq
\label{eq2.1}
{\cal D}(q^w) = \{u\in L^2(\real^n); q^w(x,D_x) u\in L^2(\real^n)\},
\endeq
and let us recall from~\cite[Section 4]{Ho95} that the operator $q^w(x,D_x)$ agrees with the graph closure of its restriction to
${\cal S}(\real^n)$,
$$
q^w(x,D_x): {\cal S}(\real^n) \rightarrow {\cal S}(\real^n).
$$
It follows, as observed in~\cite[Section 4]{Ho95}, that the operator $q^w(x,D_x)$ is maximal accretive and generates, in view of the Hille-Yosida theorem,
a strongly continuous contraction semigroup
\begeq
\label{eq2.2}
e^{-tq^w}: L^2(\real^n)\rightarrow L^2(\real^n),\quad t\geq 0.
\endeq
In our proof of Theorem \ref{theo_main}, following~\cite{HeSjSt},~\cite{HPS09},~\cite{Sj10}, we shall be concerned with the small time behavior of the semigroup
(\ref{eq2.2}) on the FBI transform side. Let
\begeq
\label{eq2.3}
Tu(x) = C\int e^{i\varphi(x,y)}u(y)\,dy,\quad x\in \comp^n,\quad C>0,
\endeq
be a metaplectic FBI-Bargmann transform, see~\cite{Sj95},~\cite{HiSj15}. Here $\varphi$ is a holomorphic quadratic form on $\comp^n_x\times \comp^n_y$, such that
$$
{\rm det}\,\varphi''_{xy}\neq 0,\quad {\rm Im}\,\varphi''_{yy}>0.
$$
Associated to $T$ there is a complex linear canonical transformation
\begeq
\label{eq2.31}
\kappa_T: \comp^{2n}\ni (y,-\varphi'_y(x,y))\mapsto (x,\varphi'_x(x,y))\in \comp^{2n}.
\endeq
We recall from~\cite[Proposition 1.1]{Sj95},~\cite[Theorem 1.3.3]{HiSj15} that if $C>0$ is suitably chosen in (\ref{eq2.3}), then $T$ is unitary,
\begeq
\label{eq2.4}
T: L^2(\real^n)\rightarrow H_{\Phi_0}(\comp^n),
\endeq
where
\begeq
\label{eq2.5}
H_{\Phi_0}(\comp^n) = {\rm Hol}(\comp^n)\cap L^2(\comp^n, e^{-2\Phi_0(x)}L(dx)),
\endeq
with
\begeq
\label{eq2.6}
\Phi_0(x) = \sup_{y\in {\bf R}^n} \left(-{\rm Im}\,\varphi(x,y)\right)
\endeq
and $L(dx)$ being the Lebesgue measure on $\comp^n$. Let us also recall from~\cite[Section 1]{Sj95},~\cite[Proposition 1.3.2]{HiSj15} that the real quadratic form
$\Phi_0$ in (\ref{eq2.6}) is strictly plurisubharmonic on $\comp^n$.

\medskip
\noindent
We have the exact Egorov property,~\cite[Proposition 1.4]{Sj95},~\cite[Theorem 1.4.2]{HiSj15},
\begeq
\label{eq2.7}
T q^w(y,D_y) u = \widetilde{q}^w(x,D_x)Tu,\quad u\in {\cal S}(\real^n),
\endeq
where $\widetilde{q}$ is a quadratic form on $\comp^{2n}$ given by
\begeq
\label{eq2.71}
\widetilde{q} = q\circ \kappa_T^{-1}.
\endeq
We refer to~\cite[Section 1]{Sj95},~\cite[Section 1.4]{HiSj15} for a discussion of the Weyl quantization in quadratic $H_{\Phi}$--spaces. Let us also recall
from~\cite[Section 1]{Sj95},~\cite[Proposition 1.3.2]{HiSj15} that the canonical transformation $\kappa_T$ maps $\real^{2n}$ bijectively onto
\begeq
\label{eq2.8}
\Lambda_{\Phi_0} = \left\{\left(x,\frac{2}{i}\frac{\partial \Phi_0}{\partial x}(x)\right); x\in \comp^n\right\}.
\endeq
Here the real linear subspace $\Lambda_{\Phi_0}\subset \comp^{2n}$ is I-Lagrangian and R-symplectic, and in particular, it is maximally totally real. The
holomorphic quadratic form $\widetilde{q}$ is therefore uniquely determined by its restriction to $\Lambda_{\Phi_0}$, and we may notice, in view of
(\ref{eq1.41}) and (\ref{eq2.71}), that
\begeq
\label{eq2.80}
{\rm Re}\,\widetilde{q}\left(x,\frac{2}{i}\frac{\partial \Phi_0}{\partial x}(x)\right)\geq 0,\quad x\in \comp^n.
\endeq

\bigskip
\noindent
Let us simplify the notation and write in what follows, $Q = \widetilde{q}^w(x,D_x)$. The operator $Q$ is a holomorphic quadratic differential operator and
we would like to study the unbounded operator $e^{tQ}$ on $H_{\Phi_0}(\comp^n)$, for $0\leq t\leq t_0$, with $t_0>0$ sufficiently small, see also~\cite{AlVi}.
To that end, let us consider the evolution problem
\begeq
\label{eq2.81}
\left(\partial_t - Q\right)u(t,x) = 0,\quad u|_{t=0} = u_0 \in H_{\Phi_0}(\comp^n),
\endeq
for $t\in [0,t_0]$, which we can solve by a geometric optics construction. Let $\phi(t,x,\eta)$ be the holomorphic quadratic form on
$\comp^n_x\times \comp^n_{\eta}$, depending smoothly on $t\in[0,t_0]$, with $t_0 > 0$ sufficiently small, and satisfying the Hamilton-Jacobi equation,
\begeq
\label{eq2.9}
i\partial_t \phi(t,x,\eta) - \widetilde{q}\left(x,\partial_x \phi(t,x,\eta)\right) = 0,\quad \phi(0,x,\eta) = x\cdot \eta.
\endeq
From the general Hamilton-Jacobi theory~\cite[Chapter 1]{DiSj}, we know that for $t\in [0,t_0]$, with $t_0 > 0$ small enough, the quadratic form
$\phi(t,x,\eta)$ can be obtained as a generating function for the complex linear canonical transformation
\begeq
\label{eq2.10}
\exp(itH_{\widetilde{q}}): \comp^{2n} \ni \left(\partial_{\eta}\phi(t,x,\eta),\eta\right) \mapsto \left(x,\partial_x \phi(t,x,\eta)\right)\in \comp^{2n}.
\endeq
Here, when $f$ is a holomorphic function on $\comp^{2n}= \comp^n_x\times \comp^n_{\xi}$, the Hamilton vector field $H_f$ of $f$ is a holomorphic vector field
given by the usual formula,
$$
H_f = \sum_{j=1}^n \left(\frac{\partial f}{\partial \xi_j} \frac{\partial}{\partial x_j} - \frac{\partial f}{\partial x_j} \frac{\partial}{\partial x_j}\right).
$$
It follows that for $0\leq t \leq t_0 \ll 1$, the solution operator $e^{tQ}$ to (\ref{eq2.81}) is given by the following quadratic Fourier integral operator in
the complex domain,
\begeq
\label{eq2.11}
e^{tQ}u(x) = \frac{1}{(2\pi)^n} \int\!\!\!\int_{\Gamma(x,t)} e^{i\left(\phi(t,x,\eta) - y\cdot \eta\right)} a(t)\, u(y)\,dy\,d\eta,\quad u\in H_{\Phi_0}(\comp^n).
\endeq
Here the amplitude $a(t)$, depending smoothly on $t$, with $a(0) = 1$, is obtained by solving a suitable transport equation, which we need not specify here.
When explaining the choice of the contour of integration $\Gamma(x,t)$ in (\ref{eq2.11}), we shall follow the discussion in Appendix B of~\cite{CaGrHiSj}, which
in turn can be viewed as a linear version of the general theory described in~\cite[Chapters 3,4]{Sj82}. Let $\Phi_t$ be a real strictly plurisubharmonic quadratic form on $\comp^n$, depending smoothly on $t\in [0,t_0]$,
for $t_0 > 0$ small enough, such that if we set
\begeq
\label{eq2.12}
\Lambda_{\Phi_t} = \left\{\left(x,\frac{2}{i}\frac{\partial \Phi_t}{\partial x}(x)\right); x\in \comp^n\right\}
\endeq
then
\begeq
\label{eq2.13}
\Lambda_{\Phi_t} = \exp\left(t \widehat{H_{i\widetilde{q}}}\right)(\Lambda_{\Phi_0}).
\endeq
Here, when $\nu$ is a vector field of type $(1,0)$ on $\comp^{2n}$, we let $\widehat{\nu} = \nu + \overline{\nu}$ be the corresponding real vector field. An
application of~\cite[Proposition B.1]{CaGrHiSj} and~\cite[Proposition B.2]{CaGrHiSj} allows us to conclude that the plurisubharmonic quadratic form
$$
\comp^n \times \comp^n \ni (y,\eta) \mapsto -{\rm Im}\,\left(\phi(t,x,\eta)-y\cdot \eta\right) + \Phi_0(y)
$$
has a unique critical point $(y_c(x,t),\eta_c(x,t))$ for each $x\in \comp^{n}$ and $t\in [0,t_0]$, which is non-degenerate of signature $(2n,2n)$. Furthermore,
we have
$$
\Phi_t(x) = {\rm vc}_{y,\eta}\left(-{\rm Im}\,\left(\phi(t,x,\eta)-y\cdot \eta\right) + \Phi_0(y)\right),
$$
where the general notation ${\rm vc}_{y,\eta}\left(\ldots\,\right)$ stands for the critical value with respect to $y$ and $\eta$ of $\left(\ldots\,\right)$.
The contour $\Gamma(x,t) \subset \comp^{2n}_{y,\eta}$ in (\ref{eq2.11}) is a so called good contour~\cite[Chapter 3]{Sj82}, which is an affine subspace of
$\comp^{2n}$ of real dimension $2n$, passing through the critical point $(y_c(x,t),\eta_c(x,t))$, and along which we have
$$
-{\rm Im}\,\left(\phi(t,x,\eta)-y\cdot \eta\right) + \Phi_0(y) - \Phi_t(x) \asymp -\left(\abs{y -y_c(x,t)}^2 + \abs{\eta - \eta_c(x,t)}^2\right).
$$
Applying ~\cite[Proposition B.3]{CaGrHiSj}, we conclude that we have a bounded operator
\begeq
\label{eq2.14}
e^{tQ}: H_{\Phi_0}(\comp^n) \rightarrow H_{\Phi_t}(\comp^n),\quad 0 \leq t \leq t_0 \ll 1,
\endeq
where similarly to (\ref{eq2.5}), we set
$$
H_{\Phi_t}(\comp^n) = {\rm Hol}(\comp^n)\cap L^2(\comp^n, e^{-2\Phi_t(x)}L(dx)).
$$

\bigskip
\noindent
Associated to the function $\Phi(t,x) = \Phi_t(x)$ is the manifold
$$
\tau = \frac{\partial \Phi}{\partial t},\quad \xi = \frac{2}{i}\frac{\partial \Phi}{\partial x},
$$
in $\real^2_{t,\tau} \times \comp^{2n}_{x,\xi}$, which is Lagrangian with respect to the real symplectic form
\begeq
\label{eq2.14.1}
d\tau \wedge dt - {\rm Im}\, \sigma,
\endeq
where
$$
\sigma = \sum_{j=1}^n d\xi_j \wedge dx_j,
$$
is the complex symplectic $(2,0)$--form on $\comp^{2n} = \comp^n_x\times \comp^n_{\xi}$. Let us also recall the general relation~\cite{HPS09},
\begeq
\label{eq2.14.2}
\widehat{H_{i\widetilde{q}}} = H_{-{\rm Re}\,\widetilde{q}}^{-{\rm Im}\,\sigma},
\endeq
where $H^{-{\rm Im}\,\sigma}_g$ is the Hamilton vector field of a function $g \in C^1(\comp^{2n},\real)$, computed with respect to the real
symplectic form $-{\rm Im}\, \sigma$. The Hamilton-Jacobi theory applied with respect to the real symplectic form in (\ref{eq2.14.1}) tells us
therefore that the function $\Phi(t,x)$ satisfies the real Hamilton-Jacobi equation,
\begeq
\label{eq2.15}
\frac{\partial \Phi}{\partial t}(t,x) - {\rm Re}\,\widetilde{q}\left(x,\frac{2}{i}\frac{\partial \Phi}{\partial x}(t,x)\right) = 0,\quad
\Phi(0,\cdot) = \Phi_0,
\endeq
for $x\in \comp^n$, $0\leq t \leq t_0\ll 1$. See also~\cite[Section 3]{Sj83} and~\cite[Section 3]{HPS09}.

\medskip
\noindent
Now (\ref{eq2.14.2}) implies that the function ${\rm Re}\, \widetilde{q}$ is constant along the flow of the Hamilton vector field $\widehat{H_{i\widetilde{q}}}$, and
it follows from (\ref{eq2.80}) and (\ref{eq2.13}) that ${\rm Re}\, \widetilde{q}|_{\Lambda_{\Phi_t}} \geq 0$. Using (\ref{eq2.15}) we conclude that
$$
\frac{\partial \Phi}{\partial t}(t,x) \geq 0,
$$
so that the function $t\mapsto \Phi_t(x)$ is increasing.

\bigskip
\noindent
{\it Remark}. Let us consider estimates for the operator norm of
$$
e^{tQ} \in {\cal L}(H_{\Phi_0}(\comp^n), H_{\Phi_t}(\comp^n)),
$$
for $0 \leq t \leq t_0$. When doing so, let us set 
$$
L^2_{\Phi_0}(\comp^n) = L^2 (\comp^n; e^{-2\Phi_0} L(dx)),
$$
and let $u\in H_{\Phi_0}(\comp^n)$ be such that
\begeq
\label{eq2.15.01}
\langle{x\rangle}^N u \in L^2_{\Phi_0}(\comp^n), 
\endeq
for all $N\in \nat$. We recall from~\cite{Sj95} that functions in $H_{\Phi_0}(\comp^n)$ satisfying (\ref{eq2.15.01}) are precisely those for which
$T^{-1}u \in {\cal S}(\real^n)$, and in particular, such functions are dense in $H_{\Phi_0}(\comp^n)$. Let us differentiate the scalar product
$$
(e^{tQ}u, e^{tQ}u)_{H_{\Phi_t}} = (u(t),u(t))_{H_{\Phi_t}}
$$
with respect to $t$.
We get
\begin{multline*}
\frac{d}{dt} (u(t),u(t))_{H_{\Phi_t}} = \\ (Qu(t),u(t))_{H_{\Phi_t}} + (u(t), Qu(t))_{H_{\Phi_t}}
- 2 \int \abs{u(t)}^2 e^{-2\Phi_t(x)}\frac{\partial \Phi_t}{\partial t}(x)\, L(dx).
\end{multline*}
Here the first two terms in the right hand side can be simplified by means of the quantization-multiplication
formula~\cite[Theorem 1.2]{Sj90},~\cite[Proposition 1.4.4]{HiSj15}, which becomes exact in the present quadratic case,
\begeq
\label{eq2.15.1}
(Qu(t),u(t))_{H_{\Phi_t}} = \int \widetilde{q}\left(x,\frac{2}{i}\frac{\partial \Phi_t}{\partial x}(x)\right) \abs{u(t)}^2 e^{-2\Phi_t(x)}\, L(dx) +
b(t)\norm{u(t)}^2_{H_{\Phi_t}},
\endeq
where $b \in C^{\infty}([0,t_0])$. We refer to~\cite[Section 2]{AlVi} for an explicit computation of $\Phi_t$ and $b(t)$, in a particular
FBI representation. Combining (\ref{eq2.15.1}) with its analog for $(u(t), Qu(t))_{H_{\Phi_t}}$ and using (\ref{eq2.15}) we obtain that
$$
\frac{d}{dt} \norm{u(t)}^2_{H_{\Phi_t}} = \left(2{\rm Re}\, b(t)\right) \norm{u(t)}^2_{H_{\Phi_t}}.
$$
Setting $B(t) = \int_0^t {\rm Re}\, b(s)\,ds \in C^{\infty}([0,t_0])$, we conclude that the operator
\begeq
\label{eq2.15.2}
e^{-B(t)} e^{tQ}: H_{\Phi_0}(\comp^n) \rightarrow H_{\Phi_t}(\comp^n),\quad 0 \leq t \leq t_0,
\endeq
is an isometry. As we shall discuss in the next section, it follows from the general theory~\cite[Chapters 3,4]{Sj82} that it is necessarily a bijection, and
therefore the map (\ref{eq2.15.2}) is unitary.

\bigskip
\noindent
Let us recall that the integer $0 \leq k_0 \leq 2n-1$ has been introduced in (\ref{eq1.7}). The following is the main result of this section.
\begin{theo}
\label{theo_sec2}
There exist $t_0 > 0$ and $C>0$ such that for all $t\in [0,t_0]$ we have
\begeq
\label{eq2.16}
\Phi_t(x) \geq \Phi_0(x) + \frac{t^{2k_0+1}}{C}\abs{x}^2,\quad x\in \comp^n.
\endeq
\end{theo}

\medskip
\noindent
When proving Theorem \ref{theo_sec2}, we shall follow the arguments of~\cite{Sj10} closely, and indeed, the following discussion can be viewed as a
straightforward adaptation of the analysis of~\cite[Sections 2,3]{Sj10} to the quadratic case. Let us introduce a $C^{\infty}$--family of linear
I-Lagrangian R-symplectic manifolds $\Lambda_t$, given by
\begeq
\label{eq2.17}
\Lambda_t = \kappa_T^{-1}\left(\Lambda_{\Phi_t}\right)\subset \comp^{2n}, \quad t \in [0,t_0],\quad 0 < t_0 \ll 1. 
\endeq
Using (\ref{eq2.71}) and (\ref{eq2.13}), we can write
$$
\Lambda_t = \kappa_T^{-1} \circ \exp(itH_{\widetilde{q}})(\Lambda_{\Phi_0}) = \kappa_T^{-1} \circ \exp(itH_{\widetilde{q}})\circ \kappa_T \left(\real^{2n}\right) =
\exp(itH_q)\left(\real^{2n}\right).
$$
Here we have identified the flow $\exp(t\widehat{H_{i\widetilde{q}}})$ of the real vector field $\widehat{H_{i\widetilde{q}}}$ with the
holomorphic flow of the holomorphic vector field $H_{i\widetilde{q}}$, restricted to small positive $t\in \real$.
The I-Lagrangian manifold $\Lambda_t\subset \comp^{2n}_{x,\xi}$ is an ${\cal O}(t)$--perturbation of $\real^{2n}$ in the sense of linear subspaces,
and the real 1-form ${\rm Im}\, \left(\xi \cdot dx\right)$ is closed, and hence exact, on $\Lambda_t$. It follows that there exists a unique real-valued
quadratic form $G_t$ on $\real^{2n}$, depending smoothly on $t\in {\rm neigh}(0,[0,\infty))$, such that $G_0 = 0$ and
\begeq
\label{eq2.18}
\Lambda_t = \kappa_T^{-1}\left(\Lambda_{\Phi_t}\right) = \{X + iH_{G_t}(X);\,\, X\in \real^{2n}\}.
\endeq
Here $H_{G_t}$ is the Hamilton vector field of $G_t$.

\bigskip
\noindent
{\it Example}. Let $q: \real^n_x\times \real^n_{\xi} \rightarrow \real$ be real-valued and positive semi-definite. Then, using the fact that $H_q = 2F$ is real,
we write for small $t\in \real$,
$$
\Lambda_t = \exp(itH_q)\left(\real^{2n}\right) = \{X + i {\rm tan}(2tF)X;\, X\in \real^{2n}\},
$$
where ${\rm tan}(2tF) = {\rm sin}(2tF) \left({\rm cos}(2tF)\right)^{-1}$. Observing that ${\rm cos}\, F$ is symmetric and ${\rm sin}\, F$ is skew-symmetric
with respect to the symplectic form $\sigma$, we see that the bilinear form
$$
\real^{2n} \times \real^{2n} \ni (X,Y)\mapsto \sigma(X,{\rm tan}\,(2tF)Y)
$$
is symmetric. It follows that the quadratic form $G_t$ in (\ref{eq2.18}) is given by
$$
G_t(X) = \frac{1}{2}\sigma(X, {\rm tan}\,(2tF)X).
$$

\bigskip
\noindent
In~\cite[Proposition 2.1]{Sj10}, it is explained how to recover the quadratic form $\Phi_t$ in (\ref{eq2.12}) from $G_t$, and we recall from this result that
\begeq
\label{eq2.19}
\Phi_{t}(x) = {\rm vc}_{(y,\eta)\in {\bf C}^n \times {\bf R}^n} \left(-{\rm Im}\, \varphi(x,y) - \eta\cdot {\rm Im}\, y + G_t({\rm Re}\, y,\eta)\right).
\endeq
Here $\varphi(x,y)$ is the phase of the FBI-Bargmann transform in (\ref{eq2.3}). For the convenience of the reader, we shall now discuss briefly 
the derivation of the formula (\ref{eq2.19}), following~\cite{Sj10}. When doing so, let us observe that the point $(y,\eta) \in \comp^n \times \real^n$ is 
a critical point of the function 
$$
\comp^n \times \real^n \ni (y,\eta) \mapsto -{\rm Im}\, \varphi(x,y) - \eta\cdot {\rm Im}\, y + G_t({\rm Re}\, y,\eta)
$$
precisely when we have 
\begeq
\label{eq2.19.001}
{\rm Im}\, y = \nabla_{\eta} G_t({\rm Re}\, y,\eta), \quad \frac{\partial}{\partial {\rm Re}\, y} {\rm Im}\,\varphi(x,y) = \nabla_y G_t({\rm Re}\,y,\eta),
\endeq
and 
\begeq
\label{eq2.19.002}
\frac{\partial}{\partial {\rm Im}\, y} {\rm Im}\,\varphi(x,y) + \eta = 0.
\endeq
In view of the Cauchy-Riemann equations, we can write 
$$
\frac{\partial}{\partial {\rm Re}\, y} {\rm Im}\,\varphi(x,y) = {\rm Im}\, \varphi'_y(x,y), \quad 
\frac{\partial}{\partial {\rm Im}\, y} {\rm Im}\,\varphi(x,y) = {\rm Re}\, \varphi'_y(x,y),
$$
and therefore, (\ref{eq2.19.001}), (\ref{eq2.19.002}) can equivalently be stated as follows, 
$$
(y,-\varphi'_y(x,y)) = ({\rm Re}\,y,\eta) + i H_{G_t}({\rm Re}\, y, \eta).
$$
We conclude that the critical value in (\ref{eq2.19}) is attained at a unique critical point $(y,\eta) = (y(x,t),\eta(x,t))\in \comp^n \times \real^n$,
which is determined by the condition that the linear canonical transformation $\kappa_T$ in (\ref{eq2.31}) maps the point
\begeq
\label{eq2.19.01}
({\rm Re}\, y,\eta) + i H_{G_t}({\rm Re}\, y,\eta) \in \Lambda_t = \kappa_T^{-1}\left(\Lambda_{\Phi_t}\right)
\endeq
to the point $(x,\varphi'_x(x,y))$, situated above $x \in \comp^n$. As verified in~\cite{Sj10}, the critical point is non-degenerate, and in order to complete
the proof of (\ref{eq2.19}), it suffices to observe that if $\Phi_t(x)$ stands for the critical value in (\ref{eq2.19}), we have, using that $\Phi_t$ is a 
critical value,   
$$
\frac{2}{i}\frac{\partial \Phi_t}{\partial x}(x) = \frac{2}{i} \frac{\partial} {\partial x} \left(-{\rm Im}\, \varphi(x,y)\right) = \varphi'_x(x,y),
$$
where $(y,\eta)$ is the corresponding critical point. 

\medskip
\noindent
Continuing to follow~\cite{Sj10}, from~\cite[Proposition 2.4]{Sj10}, let us also recall the following inversion formula for $(y,\eta)\in \real^{2n}$,
\begeq
\label{eq2.19.1}
G_t(y,\eta) = {\rm vc}_{(x,\theta)\in {\bf C}^n \times {\bf R}^n} \left({\rm Im}\, \varphi(x,y+i\theta) + \eta\cdot \theta  + \Phi_t(x)\right).
\endeq
Using that $\Phi_t \geq \Phi_0$ we conclude as in~\cite[Section 2]{Sj10}, that $G_t \geq 0$ for all $t\in [0,t_0]$, for $t_0 > 0$ sufficiently small.

\medskip
\noindent
We shall next show, following~\cite[Section 3]{Sj10}, that the real Hamilton-Jacobi equation (\ref{eq2.15}) for $\Phi_t$ implies a similar equation for $G_t$.
To this end, let $(x(t,y,\eta), \theta(t,y,\eta))$ be the critical point in (\ref{eq2.19.1}), and let us write, using the fact that $G_t$ is the
critical value in (\ref{eq2.19.1}) together with (\ref{eq2.15}),
$$
\frac{\partial G_t}{\partial t}(y,\eta) = \frac{\partial \Phi_t}{\partial t}(x(t,y,\eta)) =
{\rm Re}\,\widetilde{q}\left(x,\frac{2}{i}\frac{\partial \Phi_t}{\partial x}(x)\right)\bigg |_{x = x(t,y,\eta)}.
$$
Now as discussed above, the critical points in (\ref{eq2.19}) and (\ref{eq2.19.1}) are related to $\kappa_T$, so that
$$
\left(x(t,y,\eta), \frac{2}{i} \frac{\partial \Phi_t}{\partial x}(x(t,y,\eta))\right) = \kappa_T \left((y,\eta) + i H_{G_t}(y,\eta)\right).
$$
We obtain therefore the following equation for the smooth family of quadratic forms $G_t$,
\begeq
\frac{\partial G_t}{\partial t}(y,\eta) = {\rm Re}\, \left(q((y,\eta) + iH_{G_t}(y,\eta))\right),
\endeq
which we state as an initial value problem, for $X\in \real^{2n}$ and $t\in [0,t_0]$, for $t_0 > 0$ small enough,
\begeq
\label{eq2.20}
\frac{\partial G_t}{\partial t}(X) - {\rm Re}\,(q\left(X + iH_{G_t}(X)\right)) = 0,\quad G_0 = 0.
\endeq

\medskip
\noindent
Using the fact that $q$ is quadratic, we write
\begeq
\label{eq2.20.1}
{\rm Re}\,(q\left(X + iH_{G_t}(X)\right)) = {\rm Re}\, q(X) + H_{{\rm Im}\, q} G_t(X) - {\rm Re}\, q(H_{G_t}(X)),
\endeq
and the equation (\ref{eq2.20}) becomes
\begeq
\label{eq2.20.11}
\frac{\partial G_t}{\partial t}(X) - H_{{\rm Im}\, q} G_t(X) + {\rm Re}\, q(H_{G_t}(X)) = {\rm Re}\, q(X),\quad G_0 = 0.
\endeq
Here, as we have already observed, the quadratic form $G_t\geq 0$ is nonnegative, and therefore, for some constant $C>0$, we have for all $t\in [0,t_0]$,
\begeq
\label{eq2.20.12}
0\leq {\rm Re}\, q(H_{G_t}(X)) \leq {\cal O}(1)\abs{\nabla G_t(X)}^2 \leq C G_t(X).
\endeq
We get, using (\ref{eq2.20.11}) and (\ref{eq2.20.12}),
\begeq
\label{eq2.20.2}
\frac{\partial G_t}{\partial t}(X) - H_{{\rm Im}\, q} G_t(X) + C G_t(X) \geq {\rm Re}\, q(X),\quad G_0 = 0.
\endeq
Considering (\ref{eq2.20.2}) as a differential inequality along the integral curves of $H_{-{\rm Im}\,q}$ and using Gronwall's lemma we see that for all
$t\in [0,t_0]$ and $X\in \real^{2n}$, we have for some $C > 0$,
\begeq
\label{eq2.20.3}
G_t(\exp(tH_{-{\rm Im}\,q})(X)) \geq \frac{1}{C}\int_0^t {\rm Re}\,q \left(\exp(sH_{-{\rm Im}\,q})(X)\right)\,ds.
\endeq

\medskip
\noindent
Let us put
\begeq
\label{eq2.21}
J(t,X) = \int_0^t {\rm Re}\, q\left(\exp(sH_{{\rm Im}\, q})(X)\right)\,ds,\quad t\in [0,t_0],\quad X\in \real^{2n},
\endeq
so that $0 \leq J(t,X)$ is a quadratic form on $\real^{2n}$ with $C^{\infty}$ dependence on $t\in [0,t_0]$.
From~\cite[Proposition 2.0.1]{HPS09}, let us recall that for each $t>0$, the quadratic form $X \mapsto J(t,X)$ is positive definite on $\real^{2n}$.
The following result, which is an analog of ~\cite[Proposition 3.2]{Sj10}, is a sharpening of this basic observation.

\begin{prop}
There exist $t_0>0$ and $C>0$ such that for all $t\in [0,t_0]$ we have
\begeq
\label{eq2.22}
J(t,X) \geq \frac{t^{2k_0+1}}{C}\abs{X}^2,\quad X\in \real^{2n}.
\endeq
\end{prop}
\begin{proof}
We proceed similarly to the proof of Proposition 3.2 in~\cite{Sj10}. Let $X\in \real^{2n}$ be such that $\abs{X} = 1$ and let $j\in \nat$,
$0\leq j \leq k_0$, be such that
\begeq
\label{eq2.23}
{\rm Re}\,q\left(\exp(tH_{{\rm Im}\,q})(X)\right) = at^{2j} + {\cal O}(t^{2j+1}),\quad t\rightarrow 0,
\endeq
where $a > 0$. Here we use the property (\ref{eq1.71}). We claim that there exist $C_X > 0$, $t_X\in (0,1)$, and a neighborhood
$V_X$ of $X$ in $S^{n-1}$ such that for all $t\in (0,t_X]$ and all $Y\in V_X$, we have
\begeq
\label{eq2.24}
J(t,Y) \geq \frac{t^{2j+1}}{C_X}.
\endeq
Seeking a contradiction, let us assume that (\ref{eq2.24}) does not hold. Then there exist sequences $0 < t_{\nu} \rightarrow 0$, $Y_{\nu} \in S^{n-1}$,
$Y_{\nu} \rightarrow X$ such that
\begeq
\label{eq2.25}
\frac{J(t_{\nu},Y_{\nu})}{t_{\nu}^{2j+1}}\rightarrow 0,\quad \nu \rightarrow \infty.
\endeq
Using the fact that the function $t\mapsto J(t,X)$ is increasing, we get from (\ref{eq2.25}),
\begeq
\label{eq2.26}
\sup_{0\leq t \leq t_{\nu}} \frac{J(t,Y_{\nu})}{t_{\nu}^{2j+1}} \rightarrow 0.
\endeq
A Taylor expansion gives that
\begeq
\label{eq2.26.0}
J(t,Y_{\nu}) = a_{\nu}^{(0)} + a_{\nu}^{(1)}t + \ldots\, + a_{\nu}^{(2j+1)} t^{2j+1} + {\cal O}(t^{2j+2}),\quad t\rightarrow 0,
\endeq
and let us define
$$
0 \leq u_{\nu}(s) = \frac{J(t_{\nu}s,Y_{\nu})}{t_{\nu}^{2j+1}},\quad 0\leq s \leq 1,
$$
so that according to (\ref{eq2.26}), we have
\begeq
\label{eq2.26.1}
\sup_{0\leq s \leq 1} u_{\nu}(s) \rightarrow 0,\quad \nu \rightarrow \infty.
\endeq
On the other hand, using (\ref{eq2.26.0}), we write
\begeq
\label{eq2.26.2}
u_{\nu}(s) = p_{\nu}(s) + {\cal O}(t_{\nu} s^{2j+2}),
\endeq
where $p_{\nu}$ is a polynomial of degree $2j+1$ given by
$$
p_{\nu}(s) = \frac{a_{\nu}^{(0)}}{t_{\nu}^{2j+1}} + \frac{a_{\nu}^{(1)}}{t_{\nu}^{2j}}s + \ldots\, + a_{\nu}^{(2j+1)} s^{2j+1}.
$$
It follows from (\ref{eq2.26.1}) and (\ref{eq2.26.2}) that $p_{\nu}\rightarrow 0$ uniformly on $[0,1]$ as $\nu\rightarrow \infty$, and since all norms on a
finite-dimensional vector space are equivalent, we see that the coefficients of $p_{\nu}$ all tend to $0$ as $\nu \rightarrow \infty$. In particular,
$$
a_{\nu}^{(2j+1)} = \frac{1}{(2j+1)!} \left(\partial_t^{2j+1} J\right)(0,Y_{\nu}) \rightarrow 0.
$$
On the other hand, (\ref{eq2.23}) shows that
$$
\partial_t^{2j+1} J(0,X) = (2j)! a > 0,
$$
and this contradiction establishes the claim. In particular, we see that for all $t\in [0,t_X]$ and all $Y\in V_X$ we have
$$
J(t,Y) \geq \frac{t^{2k_0+1}}{C_X}.
$$
Covering the compact set $S^{n-1}$ by finitely many open neighborhoods of the form $V_{X_1},\ldots\, V_{X_N}$ and letting $C = {\rm max}_{1\leq j \leq N} C_{X_j}$,
$t_0 = {\rm min}_{1\leq j \leq N} t_{X_j}>0$, we conclude that for all $Y\in S^{n-1}$, we have
$$
J(t,Y) \geq \frac{t^{2k_0+1}}{C},\quad 0\leq t \leq t_0.
$$
Using the fact that $Y\mapsto J(t,Y)$ is quadratic, we conclude the proof.
\end{proof}

\medskip
\noindent
Coming back to (\ref{eq2.20.3}), we observe that the conclusion of Proposition 2.2 remains valid also for the function
$$
\widetilde{J}(t,X) = \int_0^t {\rm Re}\,q\left(\exp(sH_{{\rm -Im}\,q})(X)\right)\,ds.
$$
Indeed, the Hamilton map of the quadratic form $X \mapsto \overline{q(X)}$ is the complex conjugate $\overline{F}$ of the Hamilton map $F$ of $q$, and it
follows therefore that the assumptions of Theorem \ref{theo_main} are also valid for $\overline{q}$, with the same value of $k_0$. Combining this observation with (\ref{eq2.20.3}) we obtain
that there exist $C>0$ and $t_0 > 0$ such that the following estimate holds,
\begeq
\label{eq2.27}
G_t(X) \geq \frac{t^{2k_0+1}}{C}\abs{X}^2,\quad t\in [0,t_0],\quad X\in \real^{2n}.
\endeq
We next observe that, as explained in~\cite[Section 2]{Sj10}, the inequality (\ref{eq2.27}) together with (\ref{eq2.19}) shows that
\begeq
\label{eq2.27.1}
\Phi_t(x) \geq \Psi_t(x),\quad x\in \comp^n,\quad t\in [0,t_0],
\endeq
where
\begeq
\label{eq2.28}
\Psi_t(x) = {\rm vc}_{(y,\eta)\in {\bf C}^n \times {\bf R}^n} \left(-{\rm Im}\, \varphi(x,y) - \eta\cdot {\rm Im}\, y +
\frac{t^{2k_0+1}}{C}\abs{({\rm Re}\,y,\eta)}^2\right).
\endeq
The unique critical point in (\ref{eq2.28}) satisfies
$$
y(x,t) = y_c(x) + {\cal O}(t^{2k_0 +1}\abs{x}),\quad \eta(x,t) = -\varphi'_y(x,y_c(x)) + {\cal O}(t^{2k_0 +1}\abs{x}),
$$
where $y_c(x)\in \real^n$ is the unique point such that
$$
\Phi_0(x) = -{\rm Im}\, \varphi(x,y_c(x)).
$$
Thus, $y_c(x)$ is the unique critical point of the function
$$
\real^n \ni y \mapsto -{\rm Im}\,\varphi(x,y),
$$
so that $\varphi'_y(x,y_c(x))$ is real. From~\cite[Chapter 7]{Sj82},~\cite[Section 1.3]{HiSj15} we also recall that the map
$\comp^n \ni x\mapsto (y_c(x),-\varphi'_y(x,y_c(x))) \in \real^n\times \real^n$ is a real linear isomorphism. It follows that
\begin{multline}
\label{eq2.28.1}
\Psi_t(x) = \Phi_0(x) + \frac{t^{2k_0+1}}{C} \abs{(y_c(x),-\varphi'_y(x,y_c(x)))}^2 + {\cal O}(t^{4k_0 +2}\abs{x}^2) \\
\geq \Phi_0(x) + \frac{t^{2k_0+1}}{{\cal O}(1)}\abs{x}^2, \quad 0 \leq t\leq t_0 \ll 1,\quad x\in \comp^{n}.
\end{multline}
In view of (\ref{eq2.27.1}) and (\ref{eq2.28.1}), the proof of Theorem \ref{theo_sec2} is complete.

\medskip
\noindent
{\it Remark}. For future reference, let us notice that the arguments developed in this section apply equally well to the bounded operator
$$
e^{-tQ}: H_{\Phi_0}(\comp^n) \rightarrow H_{\Phi_0}(\comp^n), \quad t\geq 0,
$$
and allow us to conclude that for all $t\in [0,t_0]$, with $t_0 > 0$ sufficiently small, the operator $e^{-tQ}$ is bounded,
\begeq
\label{eq2.29}
e^{-tQ}: H_{\Phi_0}(\comp^n) \rightarrow H_{\widetilde{\Phi}_t}(\comp^n),
\endeq
where $\widetilde{\Phi}_t$ is a real strictly plurisubharmonic quadratic form on $\comp^n$, depending smoothly on $t\in [0,t_0]$, such that
\begeq
\label{eq2.30}
\widetilde{\Phi}_t(x) \leq \Phi_0(x) - \frac{t^{2k_0+1}}{C}\abs{x}^2,\quad x\in \comp^n,\quad t\in [0,t_0],
\endeq
for some constant $C> 0$. This result can be viewed as a sharpening of~\cite[Lemma 3.1.2]{HPS09}.

\section{Subelliptic estimates}
\setcounter{equation}{0}
In this section, we shall complete the proof of Theorem \ref{theo_main}, relying crucially on the estimate (\ref{eq2.16}). When doing so, let us recall
from~\cite[Theorem 1.2.2]{HPS09} that the spectrum of the closed densely defined quadratic operator $Q: H_{\Phi_0}(\comp^n) \rightarrow H_{\Phi_0}(\comp^n)$ is discrete,
contained in the open right half plane $\{z\in \comp;\, {\rm Re}\, z > 0\}$. Here the domain of $Q$ is given by
$$
{\cal D}(Q) = \{u\in H_{\Phi_0}(\comp^n); Qu\in H_{\Phi_0}(\comp^n)\}.
$$
We shall be concerned with deriving estimates for the bounded operator
$$
Q^{-1}: H_{\Phi_0}(\comp^n) \rightarrow H_{\Phi_0}(\comp^n),
$$
and in doing so we write
\begeq
\label{eq3.1}
Q^{-1} = \int_0^{\infty} e^{-tQ}\,dt.
\endeq
Here
$$
e^{-tQ}: H_{\Phi_0}(\comp^n) \rightarrow H_{\Phi_0}(\comp^n),\quad t\geq 0,
$$
is the strongly continuous contraction semigroup generated by $Q$, and from~\cite[Theorem 1.2.3]{HPS09} we recall that the norm of $e^{-tQ}$ on $H_{\Phi_0}(\comp^n)$
decays exponentially as $t\rightarrow +\infty$, so that the integral in (\ref{eq3.1}) converges in ${\cal L}(H_{\Phi_0}(\comp^n),H_{\Phi_0}(\comp^n))$.

\medskip
\noindent
Let $t_0 > 0$ be small enough fixed and observe that using the semigroup property we can write,
\begeq
\label{eq3.01}
Q^{-1} = \int_0^{t_0} e^{-tQ}\, dt + \int_{t_0}^{\infty} e^{-tQ}\,dt = \int_0^{t_0} e^{-tQ}\, dt + e^{-t_0 Q} Q^{-1}.
\endeq
It follows from~\cite{HPS09}, see also (\ref{eq2.29}), (\ref{eq2.30}), that there exists $\eta > 0$ such that the operator $e^{-t_0 Q}$ is bounded,
$$
e^{-t_0 Q}: H_{\Phi_0}(\comp^n) \rightarrow H_{\Phi_0 - \eta \abs{x}^2}(\comp^n),
$$
and we shall therefore begin by discussing estimates for the first term in the right hand side of (\ref{eq3.01}), given by the operator
$$
R= \int_0^{t_0} e^{-tQ}\, dt: H_{\Phi_0}(\comp^n) \rightarrow H_{\Phi_0}(\comp^n).
$$

\medskip
\noindent
Let $u$, $v$ be holomorphic functions on $\comp^n$, such that
$$
\langle{x\rangle}^N u,\quad \langle{x\rangle}^N v\in L^2_{\Phi_0}(\comp^n),
$$
for all $N\in \nat$. Consider the scalar product
\begeq
\label{eq3.2}
(R u,v)_{H_{\Phi_0}} = \int\!\!\!\int 1_{[0,t_0]}(t)(e^{-tQ}u(x))\overline{v(x)} e^{-2\Phi_0(x)}\, L(dx)\, dt,
\endeq
which can be written as follows,
\begeq
\label{eq3.3}
(R u,v)_{H_{\Phi_0}} = \int_0^{t_0} (e^{-tQ}u,v)_{H_{\Phi_0}}\,dt = \int_0^{t_0} (e^{-tQ/2}u,e^{-tQ^*/2}v)_{H_{\Phi_0}}\,dt.
\endeq
Here we have used the semigroup property of $e^{-tQ}$ and the fact that the adjoint semigroup of $e^{-tQ}$ is generated by the adjoint $Q^*$ of $Q$ in
$H_{\Phi_0}(\comp^n)$, see ~\cite[Corollary 1.10.6]{Pazy}. From~\cite[Section 4]{Ho95},~\cite[Section 3]{HPS13}, we recall that $Q^*$ is a closed densely
defined holomorphic quadratic differential operator on $H_{\Phi_0}(\comp^n)$ such that
$$
Q^* = T\overline{q}^w(y,D_y) T^{-1},
$$
where the domain of $\overline{q}^w(y,D_y)$ is given by $\{u\in L^2(\real^n);\, \overline{q}^w(y,D_y) u\in L^2(\real^n)\}$. Furthermore, as already observed in
Section 2, the quadratic form $X\mapsto \overline{q(X)}$ also satisfies the assumptions of Theorem \ref{theo_main}, with the same value of $k_0$.

\medskip
\noindent
Using (\ref{eq3.3}), we get
\begeq
\label{eq3.3.1}
(R u,v)_{H_{\Phi_0}} = \int\!\!\!\int 1_{[0,t_0]}(t) \left(e^{-tQ/2}u(x)\right) \overline{\left(e^{-tQ^*/2}v(x)\right)} e^{-2\Phi_0(x)}\, L(dx)\, dt,
\endeq
and by the Cauchy-Schwarz inequality, we conclude that
$$
\abs{(R u,v)_{H_{\Phi_0}}}^2
$$
does not exceed the product
\begin{multline}
\label{eq3.4}
\left(\int\!\!\!\int 1_{[0,t_0]}(t) \abs{e^{-tQ/2}u(x)}^2 e^{-2\Phi_0(x)}\, L(dx)\, dt\right)\times \\
\left(\int\!\!\!\int 1_{[0,t_0]}(t) \abs{e^{-tQ^*/2}v(x)}^2 e^{-2\Phi_0(x)}\, L(dx)\, dt\right).
\end{multline}

\medskip
\noindent
When estimating (\ref{eq3.4}), we observe that according to the discussion in Section 2, we have a bounded quadratic elliptic Fourier integral operator in the
complex domain,
\begeq
\label{eq3.5}
e^{tQ/2}: H_{\Phi_0}(\comp^n) \rightarrow H_{\Phi_{t/2}}(\comp^n), \quad 0\leq t\leq t_0,
\endeq
where the strictly plurisubharmonic quadratic form $\Phi_t$ is given by (\ref{eq2.13}), (\ref{eq2.15}). An application of~\cite[Proposition B.4]{CaGrHiSj}
allows us to conclude that the operator in (\ref{eq3.5}) has a continuous two-sided inverse,
\begeq
\label{eq3.51}
e^{-tQ/2}: H_{\Phi_{t/2}}(\comp^n) \rightarrow H_{\Phi_0}(\comp^n),\quad 0\leq t \leq t_0,
\endeq
with the operator norm $\norm{e^{-tQ/2}}_{{\cal L}(H_{\Phi_{t/2}}({\bf C}^n), H_{\Phi_0}({\bf C}^n))}$ uniformly bounded, for $t\in [0,t_0]$. Let us write,
when $t\in [0,t_0]$, with $t_0 > 0$ sufficiently small,
\begeq
\label{eq3.6}
\int\abs{e^{-tQ/2}u(x)}^2 e^{-2\Phi_0(x)}\, L(dx) \leq C(t_0) \int\abs{u(x)}^2 e^{-2\Phi_{t/2}(x)} \, L(dx),
\endeq
for some constant $C(t_0)>0$. Integrating (\ref{eq3.6}) with respect to $t\in [0,t_0]$ and using Theorem \ref{theo_sec2}, we conclude that the first factor in (\ref{eq3.4})
does not exceed a constant times
\begeq
\label{eq3.61}
\int\!\!\!\int 1_{[0,t_0]}(t) \abs{u(x)}^2 e^{-2\Phi_0(x)} e^{-t^{2k_0+1}\abs{x}^2/C}\, L(dx)\,dt,
\endeq
for some $C > 0$. Here we can carry out the $t$--integration first, and to this end, we observe that
$$
\int_0^{t_0} e^{-t^{2k_0+1}\abs{x}^2/C}\,dt = {\cal O}(1),\quad \wrtext{for}\quad \abs{x}\leq 1,
$$
while for $\abs{x}\geq 1$, we get by a change of variables,
$$
\int_0^{t_0} e^{-t^{2k_0+1}\abs{x}^2/C}\,dt \leq \frac{1}{(2k_0+1)\abs{x}^{2/(2k_0+1)}} \int_0^{\infty} e^{-y/C} y^{-\frac{2k_0}{2k_0+1}}\,dy.
$$
Therefore,
\begeq
\label{eq3.62}
\int_0^{t_0} e^{-t^{2k_0+1}\abs{x}^2/C}\,dt \leq \frac{{\cal O}(1)}{(1 + \abs{x}^2)^{1/(2k_0+1)}},\quad x\in \comp^n,
\endeq
and combining (\ref{eq3.61}) and (\ref{eq3.62}), we obtain the following bound for the first factor in (\ref{eq3.4}),
\begin{multline}
\label{eq3.63}
\int\!\!\!\int 1_{[0,t_0]}(t) \abs{e^{-tQ/2}u(x)}^2 e^{-2\Phi_0(x)}\, L(dx)\, dt \\
\leq {\cal O}(1) \int (1+\abs{x}^2)^{-1/(2k_0+1)}\abs{u(x)}^2 e^{-2\Phi_0(x)}\, L(dx).
\end{multline}

\medskip
\noindent
To estimate the second factor in (\ref{eq3.4}), we observe that all of the analysis developed in Section 2 can be applied also to the operator $e^{tQ^*}$, and
in particular, we have the bounded operator
\begeq
\label{eq3.64}
e^{tQ^*}: H_{\Phi_0}(\comp^n) \rightarrow H_{\widehat{\Phi}_t}(\comp^n),\quad 0 \leq t \leq t_0,
\endeq
where, similarly to (\ref{eq2.16}), the quadratic form $\widehat{\Phi}_t$ satisfies
\begeq
\label{eq3.65}
\widehat{\Phi}_t(x) \geq \Phi_0(x) + \frac{t^{2k_0+1}}{C}\abs{x}^2,\quad x\in \comp^n,\quad t\in [0,t_0].
\endeq
Combining (\ref{eq3.3.1}), (\ref{eq3.4}), (\ref{eq3.63}), as well as the analog of the latter estimate involving $e^{-tQ^*/2}$, we obtain the estimate
\begeq
\label{eq3.7}
\abs{(Ru,v)_{H_{\Phi_0}}}\leq C \norm{\langle{x\rangle}^{-\delta}u}_{L^2_{\Phi_0}} \norm{\langle{x\rangle}^{-\delta}v}_{L^2_{\Phi_0}},
\endeq
where
\begeq
\label{eq3.71}
\delta = \frac{1}{2k_0+1}.
\endeq
Coming back to (\ref{eq3.01}), we shall next consider the contribution of the term $e^{-t_0 Q} Q^{-1}$, and to that end we write as before,
\begeq
\label{eq3.8}
(e^{-t_0 Q} Q^{-1} u,v)_{H_{\Phi_0}} = (e^{-t_0 Q/2} Q^{-1}u, e^{-t_0 Q^*/2}v)_{H_{\Phi_0}} = (Q^{-1}e^{-t_0 Q/2}u, e^{-t_0 Q^*/2}v)_{H_{\Phi_0}},
\endeq
using also the fact that the bounded operators $e^{-tQ}$ and $Q^{-1}$ commute. Applying the Cauchy-Schwarz inequality and the fact that $Q^{-1}$ is bounded on
$H_{\Phi_0}(\comp^n)$, we get
\begeq
\label{eq3.9}
\abs{(e^{-t_0 Q} Q^{-1} u,v)_{H_{\Phi_0}}} \leq {\cal O}(1) \norm{ e^{-t_0 Q/2}u}_{H_{\Phi_0}}\, \norm{ e^{-t_0 Q^*/2}v}_{H_{\Phi_0}}.
\endeq
An application of (\ref{eq3.51}) together with Theorem \ref{theo_sec2}, along with their natural analogs for the adjoint semigroup $e^{-tQ^*}$,
see (\ref{eq3.64}), (\ref{eq3.65}), allows us then to conclude that
\begeq
\label{eq3.10}
\abs{(e^{-t_0 Q} Q^{-1} u,v)_{H_{\Phi_0}}} \leq {\cal O}(1) \norm{u e^{-\eta\abs{x}^2}}_{L^2_{\Phi_0}} \norm{v e^{-\eta\abs{x}^2}}_{L^2_{\Phi_0}},
\endeq
for some $\eta > 0$. Combining (\ref{eq3.01}), (\ref{eq3.7}), and (\ref{eq3.10}), we obtain the basic estimate
\begeq
\label{eq3.11}
\abs{(Q^{-1}u,v)_{H_{\Phi_0}}}\leq C \norm{\langle{x\rangle}^{-\delta}u}_{L^2_{\Phi_0}} \norm{\langle{x\rangle}^{-\delta}v}_{L^2_{\Phi_0}},
\endeq
where the factor $0 < \delta < 1$ has been defined in (\ref{eq3.71}). Let us also recall that here $u,v$ are holomorphic and such that
$\langle{x\rangle}^N u$, $\langle{x\rangle}^N v \in L^2_{\Phi_0}(\comp^n)$ for all $N\in  \nat$.

\bigskip
\noindent
Next, we shall pass to Toeplitz operators, and to that end we would like to replace $\langle{x\rangle}^{-\delta}u$ in (\ref{eq3.11}) by 
$\Pi \left(\langle{x\rangle}^{-\delta}u\right)$, and similarly for $\langle{x\rangle}^{-\delta}v$, where
$$
\Pi: L^2_{\Phi_0}(\comp^n) \rightarrow H_{\Phi_0}(\comp^n)
$$
is the orthogonal projection. Here, given $q\in L^{\infty}(\comp^n)$, we introduce the Toeplitz operator with the symbol $q$, acting on 
$H_{\Phi_0}(\comp^n)$, see~\cite[Section 1]{Sj95},~\cite[Chapter 13]{Zworski},
\begeq
\label{eq3.11.1}
{\rm Top}(q) u = \Pi (qu),\quad u \in H_{\Phi_0}(\comp^n).
\endeq
It will be convenient to rewrite (\ref{eq3.11}) as follows,
\begeq
\label{eq3.12}
\abs{(Q^{-1}u,v)_{H_{\Phi_0}}}\leq C(\lambda) \norm{(\lambda + \abs{x}^2)^{-\delta/2}u}_{L^2_{\Phi_0}} \norm{(\lambda + \abs{x}^2)^{-\delta/2}v}_{L^2_{\Phi_0}},
\endeq
where $\lambda \geq 1$ is to be chosen sufficiently large but fixed, and the constant $C(\lambda) > 0$ depends on $\lambda$. Using that $\Pi u =u$, we get
\begeq
\label{eq3.12.1}
\norm{(\lambda + \abs{x}^2)^{-\delta/2}u}_{L^2_{\Phi_0}} \leq \norm{\Pi (\lambda + \abs{x}^2)^{-\delta/2}u}_{H_{\Phi_0}} +
\norm{[\Pi,(\lambda + \abs{x}^2)^{-\delta/2}]u}_{L^2_{\Phi_0}}.
\endeq
Here we can write $\Pi (\lambda + \abs{x}^2)^{-\delta/2} u = {\rm Top}((\lambda + \abs{x}^2)^{-\delta/2})u$, see (\ref{eq3.11.1}). We next claim that
\begeq
\label{eq3.13}
K := [\Pi, (\lambda + \abs{x}^2)^{-\delta/2}](\lambda + \abs{x}^2)^{(\delta +1)/2} = {\cal O}(1): L^2_{\Phi_0}(\comp^n) \rightarrow L^2_{\Phi_0}(\comp^n),
\endeq
uniformly in $\lambda \geq 1$. When verifying the claim, let us recall from~\cite[Section 1]{Sj95},~\cite[Proposition 1.3.4]{HiSj15}, that the operator
$\Pi$ is given by
$$
\Pi u(x) = C \int e^{2\psi_0(x,\overline{y})} u(y) e^{-2\Phi_0(y)}\, L(dy),\quad C > 0,
$$
where $\psi_0(x,y)$ is the unique holomorphic quadratic form on $\comp^n_x \times \comp^n_y$, such that $\psi_0(x,\overline{x}) = \Phi_0(x)$. We have the
basic property,
\begeq
\label{3.13.1}
2{\rm Re}\, \psi_0(x,\overline{y}) - \Phi_0(x) - \Phi_0(y) \sim -\abs{x-y}^2,
\endeq
on $\comp^n_x \times \comp^n_y$, reflecting the strict plurisubharmonicity of $\Phi_0$, see~\cite[Section 1.3]{HiSj15}.
Letting $\Pi(x,y)$ stand for the integral kernel of $\Pi$, we have therefore,
\begeq
\label{eq3.13.2}
\abs{\Pi(x,y)}e^{\Phi_0(y) - \Phi_0(x)} \leq {\cal O}(1) e^{-\abs{x-y}^2/C},\quad C > 0.
\endeq

\medskip
\noindent
The integral kernel of the operator $K$ in (\ref{eq3.13}) is given by
\begeq
\label{eq3.14}
K(x,y) = \Pi(x,y)\left((\lambda + \abs{y}^2)^{-\delta/2} - (\lambda + \abs{x}^2)^{-\delta/2}\right) (\lambda + \abs{y}^2)^{(\delta +1)/2},
\endeq
and when estimating it, we observe that
\begeq
\label{eq3.14.1}
\abs{(\lambda + \abs{y}^2)^{-\delta/2} - (\lambda + \abs{x}^2)^{-\delta/2}} \leq {\cal O}(1)\abs{x-y}
\int_0^1 \frac{dt}{\left(\lambda + \abs{tx + (1-t)y}^2\right)^{(\delta +1)/2}},
\endeq
where ${\cal O}(1)$ is uniform in $\lambda \geq 1$. The standard inequality
$$
\frac{1 + \abs{x}}{1 + \abs{y}} \leq 1 + \abs{x-y}
$$
shows next that uniformly in $t\in [0,1]$ and $\lambda \geq 1$, we have
\begeq
\label{eq3.14.2}
\left(\frac{\lambda + \abs{y}^2}{\lambda + \abs{tx + (1-t)y}^2}\right)^{1/2} \leq 2\left(1 + \abs{x-y}\right).
\endeq
Combining (\ref{eq3.13.2}), (\ref{eq3.14.1}), and (\ref{eq3.14.2}), we conclude that the absolute value of the reduced kernel
$e^{-\Phi_0(x)} K(x,y) e^{\Phi_0(y)}$ of the operator $K$ in (\ref{eq3.13})
does not exceed
$$
{\cal O}(1)e^{-\abs{x-y}^2/C} \abs{x-y}\left(1+\abs{x-y}\right)^{\delta +1},
$$
uniformly in $\lambda \geq 1$. An application of Schur's lemma shows that (\ref{eq3.13}) holds.

\medskip
\noindent
Combining (\ref{eq3.12.1}) and (\ref{eq3.13}), we get
$$
\norm{(\lambda + \abs{x}^2)^{-\delta/2}u}_{L^2_{\Phi_0}} \leq \norm{{\rm Top}((\lambda + \abs{x}^2)^{-\delta/2})u}_{H_{\Phi_0}} +
{\cal O}(\lambda^{-1/2})\norm{(\lambda + \abs{x}^2)^{-\delta/2}u}_{L^2_{\Phi_0}},
$$
and choosing $\lambda > 1$ sufficiently large but fixed, we conclude that
\begeq
\label{eq3.15}
\norm{(\lambda + \abs{x}^2)^{-\delta/2}u}_{L^2_{\Phi_0}} \leq {\cal O}(1) \norm{{\rm Top}((\lambda + \abs{x}^2)^{-\delta/2})u}_{H_{\Phi_0}}.
\endeq
The parameter $\lambda$ will be kept fixed from now on and the dependence on $\lambda$ will not be indicated explicitly. Injecting the estimate
(\ref{eq3.15})  back into (\ref{eq3.12}), we obtain that
\begeq
\label{eq3.16}
\abs{(Q^{-1}u,v)_{H_{\Phi_0}}}\leq C \norm{{\rm Top}((\lambda + \abs{x}^2)^{-\delta/2})u}_{H_{\Phi_0}}
\norm{{\rm Top}((\lambda + \abs{x}^2)^{-\delta/2})v}_{H_{\Phi_0}}.
\endeq

\bigskip
\noindent
We shall now return to the real side by undoing the FBI transform. When doing so, let us recall from~\cite[Section 1]{Sj95} as well as
from~\cite[Theorems 13.9, 13.10]{Zworski}, that there exists $a\in C^{\infty}(\real^{2n})$ such that for all $\alpha \in \nat^{2n}$, we have
\begeq
\label{eq3.16.01}
\partial^{\alpha}a(X) = {\cal O}_{\alpha}(1)m(X),\quad m(X) = \langle{X\rangle}^{-\delta},
\endeq
and such that
\begeq
\label{eq3.16.1}
{\rm Top}((\lambda + \abs{x}^2)^{-\delta/2}) = T a^w(y,D_y) T^{-1}.
\endeq
Indeed, it follows from~\cite[Section 1]{Sj95} that we have (\ref{eq3.16.1}) with $a = b \circ \kappa_T$, where $b \in S(\Lambda_{\Phi_0})$ is such that
$$
b\left(x,\frac{2}{i}\frac{\partial \Phi_0}{\partial x}(x)\right) =
\left(\exp\left(\frac{1}{4} \left(\partial_x \partial_{\bar x} \Phi_0\right)^{-1}\partial_x \cdot \partial_{\bar x}\right)q\right)(x),\quad
q(x) = (\lambda + \abs{x}^2)^{-\delta/2}.
$$
It follows that when $u,v\in {\cal S}(\real^n)$, we get from (\ref{eq3.16}),
\begeq
\label{eq3.17}
\abs{((q^{w})^{-1}u,v)_{L^2}}\leq C \norm{a^w(x,D_x) u}_{L^2} \norm{a^w(x,D_x) v}_{L^2}.
\endeq

\bigskip
\noindent
It will be convenient to rewrite (\ref{eq3.17}) using the positive selfadjoint operator
$\Lambda = \left(1 + x^2 + D_x^2\right)^{1/2}$ on $L^2(\real^n)$, defined by means of the functional calculus for the harmonic oscillator. To that end, let us
set
$$
\Gamma = \frac{dx^2 + d\xi^2}{\langle{X\rangle}^2},\quad X = (x,\xi)\in \real^{2n},
$$
and introduce the corresponding symbol class
$$
S(\langle{X\rangle}^m, \Gamma) = \{a\in C^{\infty}(\real^{2n}); \,\,\abs{\partial^{\alpha} a(X)} \leq C_{\alpha} \langle{X\rangle}^{m-\abs{\alpha}},\,\,\alpha \in \nat^{2n}\},
\quad m\in \real.
$$
Let us recall from~\cite[Theorem 1.11.1]{Hel} that
\begeq
\label{eq3.17.1}
\Lambda^r \in {\rm Op}^w\left(S(\langle{X\rangle}^r, \Gamma)\right),\quad r\in \real.
\endeq
Using (\ref{eq3.16.01}), (\ref{eq3.17.1}), and the Calder\'on-Vaillancourt theorem, we conclude that the operator
$$
a^w(x,D_x)\Lambda^{\delta}: L^2(\real^n)\rightarrow L^2(\real^n)
$$
is bounded. We get therefore from (\ref{eq3.17}), for $u,v\in {\cal S}(\real^n)$,
\begeq
\label{eq3.18}
\abs{((q^{w})^{-1}u,v)_{L^2}}\leq C \norm{\Lambda^{-\delta}u}_{L^2} \norm{\Lambda^{-\delta}v}_{L^2}.
\endeq

\medskip
\noindent
Using that $\Lambda^{\delta}$ is a bijection on ${\cal S}(\real^n)$ and rewriting (\ref{eq3.18}) in the form
\begeq
\label{eq3.18.1}
\abs{((q^{w})^{-1}u,\Lambda^{\delta} v)_{L^2}}\leq C \norm{\Lambda^{-\delta}u}_{L^2} \norm{v}_{L^2},\quad u,v\in {\cal S}(\real^n),
\endeq
we observe that (\ref{eq3.18.1}) extends to all $u\in L^2(\real^n)$ and $v\in {\cal D}(\Lambda^{\delta})\subseteq L^2(\real^n)$, since ${\cal S}(\real^n)$ is dense
in ${\cal D}(\Lambda^{\delta})$ with respect to the graph norm. The estimate
$$
\abs{((q^{w})^{-1}u,\Lambda^{\delta} v)_{L^2}}\leq C \norm{\Lambda^{-\delta}u}_{L^2} \norm{v}_{L^2},\quad u\in L^2(\real^n),\,\, v\in {\cal D}(\Lambda^{\delta})
$$
implies that ${\cal D}(q^w)\subset {\cal D}(\Lambda^{\delta})$, with
\begeq
\label{eq3.19}
\norm{\Lambda^{\delta} (q^w)^{-1}u}_{L^2} \leq C\norm{\Lambda^{-\delta}u}_{L^2},\quad u\in L^2(\real^n).
\endeq

\medskip
\noindent
Using the bound (\ref{eq3.19}), it is now easy to conclude the proof of Theorem \ref{theo_main}. It follows from (\ref{eq3.19}) that
\begeq
\label{eq3.20}
\norm{\Lambda^{\delta} v}_{L^2} \leq C\norm{\Lambda^{-\delta}q^w v}_{L^2},\quad v\in {\cal S}(\real^n),
\endeq
and taking $v = \Lambda^{\delta}u$, $u\in {\cal S}(\real^n)$, we obtain in view of (\ref{eq3.20}),
\begeq
\label{eq3.21}
\norm{\Lambda^{2\delta} u}_{L^2} \leq C\left(\norm{q^w u}_{L^2} + \norm{\Lambda^{-\delta}[q^w, \Lambda^{\delta}]u}_{L^2}\right)
\leq C\left(\norm{q^w u}_{L^2} + \norm{u}_{L^2}\right).
\endeq
Here we have used that since $q$ is quadratic, the Weyl symbol of the operator $[q^w,\Lambda^{\delta}]$ is of the form
$i^{-1}H_q b\in S(\langle{X\rangle}^{\delta}, \Gamma)$, for some $b\in S(\langle{X\rangle}^{\delta},\Gamma)$, and therefore the operator
$\Lambda^{-\delta}[q^w, \Lambda^{\delta}] \in {\rm Op}^w(S(1,\Gamma))$ is bounded on $L^2(\real^n)$.
The proof of Theorem 1.1 is complete.

\bigskip
\noindent
{\it Remark}. The arguments developed in this section allow us to establish the following subelliptic resolvent estimate, generalizing (\ref{eq3.21}): there
exists a constant $C>0$ such that for all $u\in {\cal D}(q^w)$ and all $\lambda \in \real$, we have
\begeq
\label{eq3.22}
\norm{\Lambda^{2\delta} u}_{L^2} \leq C\left(\norm{(q^w - i\lambda) u}_{L^2} + \norm{u}_{L^2}\right).
\endeq
Indeed, in order to obtain (\ref{eq3.22}), we inspect the arguments of the present section, concluding that everything works as above, leading to (\ref{eq3.22}),
provided that we take as our starting point the representation of the resolvent along the imaginary axis as the Fourier transform of the semigroup,
$$
(Q - i \lambda)^{-1} = \int_0^{\infty} e^{-t(Q-i\lambda)}\, dt,\quad \lambda \in \real,
$$
and observe that thanks to the exponential decay of the norm of the semigroup $e^{-tQ}$, as $t\rightarrow \infty$, established in~\cite[Theorem 1.2.3]{HPS09},
we have
$$
\norm{(Q-i\lambda)^{-1}}_{{\cal L}(H_{\Phi_0}({\bf C}^n),H_{\Phi_0}({\bf C}^n))} \leq {\cal O}(1),
$$
uniformly in $\lambda \in \real$. As explained in~\cite[Section 3]{Nier} and~\cite[Subsection 3.2]{OPaPS}, the subelliptic estimate (\ref{eq3.22}) leads
directly to some accurate resolvent estimates for $q^w$, in parabolic regions near the imaginary axis.

\section{Smoothing estimates for the semigroup}
\setcounter{equation}{0}
In this section, we shall establish Theorem \ref{theo_main2}. When doing so, we shall first continue to work on the FBI transform side, and let us recall from
(\ref{eq2.29}), (\ref{eq2.30}), that we have a bounded operator
\begeq
\label{eq4.1}
e^{-tQ}: H_{\Phi_0}(\comp^n) \rightarrow H_{\widetilde{\Phi}_t}(\comp^n),
\endeq
for all $t\in [0,t_0]$, with $t_0 > 0$ sufficiently small, with the operator norm in (\ref{eq4.1}) being bounded uniformly in $t\in [0,t_0]$. Here
$\widetilde{\Phi}_t$ is a real strictly plurisubharmonic quadratic form on $\comp^n$ such that
\begeq
\label{eq4.2}
\widetilde{\Phi}_t(x) \leq \Phi_0(x) - \frac{t^{2k_0+1}}{C}\abs{x}^2,\quad x\in \comp^n,\quad t\in [0,t_0].
\endeq
Let $p_0(x,\xi) = x^2 + \xi^2$ be the symbol of the harmonic oscillator, and as in (\ref{eq2.71}), set $q_0 = p_0 \circ \kappa_T^{-1}$.
In what follows, we shall rely only on the observation that $q_0$ is a holomorphic quadratic form on $\comp^{2n}$ such that its restriction to
$\Lambda_{\Phi_0}$ is real positive definite. Let us set $Q_0 = q_0^w(x,D_x)$. The quadratic differential operator $Q_0$ is selfadjoint on
$H_{\Phi_0}(\comp^n)$, with the domain
$$
{\cal D}(Q_0) = \{u\in H_{\Phi_0}(\comp^n);\, (1+\abs{x}^2)u\in L^2_{\Phi}(\comp^n)\}.
$$
Proceeding similarly to the discussion in Section 2, we shall now consider the operator $e^{sQ_0}$, for
$s\in [0,s_0]$, with $s_0 > 0$ small enough, acting on $H_{\widetilde{\Phi}_t}(\comp^n)$. Studying the evolution problem
$$
\left(\partial_s - Q_0\right)u(s,x)=0,\quad u|_{s=0} = u_0\in H_{\widetilde{\Phi}_t}(\comp^n),
$$
and arguing as in Section 2, we see that for $s\in [0,s_0]$, with $s_0 > 0$ small enough, and all $t\in [0,t_0]$, the operator $e^{s Q_0}$ is bounded
\begeq
\label{eq4.3}
e^{s Q_0}: H_{\widetilde{\Phi}_t}(\comp^n) \rightarrow H_{\widetilde{\Phi}_{t,s}}(\comp^n).
\endeq
Here $\widetilde{\Phi}_{t,s}$ is a strictly plurisubharmonic quadratic form on $\comp^n$, depending smoothly on $t\geq 0$ and $s\geq 0$ small enough, such that
\begeq
\label{eq4.4}
\frac{\partial \widetilde{\Phi}_{t,s}}{\partial s}(x) - {\rm Re}\,q_0\left(x,\frac{2}{i}\frac{\partial \widetilde{\Phi}_{t,s}}{\partial x}(x)\right) = 0,\quad
\widetilde{\Phi}|_{t,s=0} = \widetilde{\Phi}_t.
\endeq
It follows that
\begeq
\label{eq4.5}
\widetilde{\Phi}_{t,s}(x) = \widetilde{\Phi}_t(x) + {\cal O}(s\abs{x}^2),
\endeq
uniformly for $t\in [0,t_0]$, with $t_0 > 0$ sufficiently small. Choosing
\begeq
\label{eq4.6}
s = s(t) = \frac{t^{2k_0+1}}{C_0},
\endeq
where the constant $C_0$ is large enough, we conclude, in view of (\ref{eq4.2}), (\ref{eq4.3}), and (\ref{eq4.5}), that for all
$t\in [0,t_0]$, the operator $e^{s(t) Q_0}$ is bounded,
\begeq
\label{eq4.7}
e^{s(t) Q_0}: H_{\widetilde{\Phi}_t}(\comp^n) \rightarrow H_{\Phi_0}(\comp^n).
\endeq
Combining this observation with (\ref{eq4.1}), we conclude that for all $t\in [0,t_0]$, the composition
\begeq
\label{eq4.8}
e^{s(t) Q_0} e^{-t Q}: H_{\Phi_0}(\comp^n) \rightarrow H_{\Phi_0}(\comp^n),
\endeq
is bounded, with the operator norm being bounded uniformly, for $t\in [0,t_0]$. For future reference, coming back to the real side,
let us summarize the discussion so far in the following result.

\begin{prop}
Let $q: \real^n_x \times \real^n_{\xi} \rightarrow \comp$ be a quadratic form with ${\rm Re}\,q \geq 0$, such that {\rm (\ref{eq1.51})} holds, and let us
define $k_0 \in \nat$ as in {\rm (\ref{eq1.7})}. There exist $C_0>0$ and $t_0 > 0$ such that for all $t\in [0,t_0]$, we have a bounded operator
$$
e^{\frac{t^{2k_0+1}}{C_0}(D_x^2 + x^2)} e^{-tq^w}: L^2(\real^n) \rightarrow L^2(\real^n),
$$
with the norm in ${\cal L}(L^2(\real^n),L^2(\real^n))$ bounded uniformly, for $t\in [0,t_0]$.
\end{prop}

\medskip
\noindent
{\it Remark}. The discussion in Section 2 shows that the analogs of (\ref{eq4.1}), (\ref{eq4.2}) are also valid for the adjoint semigroup,
$$
e^{-tQ^*}: H_{\Phi_0}(\comp^n) \rightarrow H_{\Phi_0}(\comp^n),\quad t\in [0,t_0],
$$
for $t_0 > 0$ small enough, and therefore, an application of Proposition 4.1 shows that we have a bounded operator
$$
e^{\frac{t^{2k_0+1}}{C_0}(D_x^2 + x^2)} e^{-t\overline{q}^w}: L^2(\real^n) \rightarrow L^2(\real^n),
$$
uniformly for $t\in [0,t_0]$. Let us introduce the domain ${\cal D} = {\cal D}(e^{\frac{t^{2k_0+1}}{C_0}(D_x^2 + x^2)}) \subseteq L^2(\real^n)$ of the
unbounded selfadjoint operator $e^{\frac{t^{2k_0+1}}{C_0}(D_x^2 + x^2)}$ on $L^2(\real^n)$, and notice that $v\in L^2(\real^n)$ belongs to ${\cal D}$
precisely when
$$
(v,\psi_{\alpha})_{L^2}\, \exp\left(\frac{2t^{2k_0+1}}{C_0}\abs{\alpha}\right) \in l^2(\nat^n).
$$
Here $\psi_{\alpha}$ are the Hermite functions, see (\ref{eq4.11}) below. In particular, we notice that ${\cal D}$ is dense in $L^2(\real^n)$, since
it contains all finite sums of the Hermite functions. When $u\in L^2(\real^n)$ and $v\in {\cal D}$, we write for $t\in [0,t_0]$,
\begeq
\label{eq4.81}
\abs{(e^{\frac{t^{2k_0+1}}{C_0}(D_x^2 + x^2)} e^{-t\overline{q}^w}u,v)_{L^2}} \leq {\cal O}(1)\norm{u}_{L^2}\norm{v}_{L^2}.
\endeq
The scalar product in the left hand side in (\ref{eq4.81}) is equal to
$$
(u, e^{-tq^w} e^{\frac{t^{2k_0+1}}{C_0}(D_x^2 + x^2)} v)_{L^2},
$$
and therefore we get
$$
\norm{e^{-tq^w} e^{\frac{t^{2k_0+1}}{C_0}(D_x^2 + x^2)} v}_{L^2} \leq {\cal O}(1)\norm{v}_{L^2},\quad t\in [0,t_0].
$$
Using the density of ${\cal D}$ in $L^2(\real^n)$, we conclude that the following operator is bounded on $L^2(\real^n)$, uniformly for $t\in [0,t_0]$,
$$
e^{-tq^w} e^{\frac{t^{2k_0+1}}{C_0}(D_x^2 + x^2)}: L^2(\real^n) \rightarrow L^2(\real^n).
$$

\bigskip
\noindent
Returning to (\ref{eq4.8}), we shall now estimate the norm of the operator $Q_0^N e^{-tQ}$, viewed as a bounded operator on $H_{\Phi_0}(\comp^n)$,
when $N\in \nat$. Writing
$$
Q_0^N e^{-tQ} = Q_0^N e^{-s(t) Q_0} e^{s(t) Q_0} e^{-t Q},
$$
where the operators involved are quadratic Fourier integral operators in the complex domain, and recalling the uniform boundedness of (\ref{eq4.8}),
we obtain that there exists a constant $C > 0$ such that for all $t\in [0,t_0]$ we have
\begeq
\label{eq4.9}
\norm{Q_0^N e^{-tQ}}_{{\cal L}(H_{\Phi_0}({\bf C}^n), H_{\Phi_0}({\bf C}^n))} \leq C
\norm{Q_0^N e^{-s(t)Q_0}}_{{\cal L}(H_{\Phi_0}({\bf C}^n), H_{\Phi_0}({\bf C}^n))}.
\endeq
By the selfadjoint functional calculus, we have
\begeq
\label{eq4.10}
\norm{Q_0^N e^{-s(t) Q_0}}_{{\cal L}(H_{\Phi_0}({\bf C}^n), H_{\Phi_0}({\bf C}^n))} \leq {\rm sup}_{\lambda \geq 0}
\left(\lambda^N e^{-s(t) \lambda}\right) \leq \frac{N!}{s(t)^N},
\endeq
and we get the smoothing estimate
$$
\norm{Q_0^N e^{-tQ}}_{{\cal L}(H_{\Phi_0}({\bf C}^n), H_{\Phi_0}({\bf C}^n))} \leq \frac{C^{N+1}N!}{t^{(2k_0+1)N}},\quad 0 < t \leq t_0,
$$
valid for some $C>0$ and all $N\in \nat$. Undoing the FBI transform and coming back to the $L^2(\real^n)$--side, we see that we have proved the bound
(\ref{eq1.9}) in Theorem \ref{theo_main2}.

\medskip
\noindent
{\it Remark}. The work~\cite{AlVi} also examines the semigroup $e^{-tq^w}$ on the FBI transform side, but focusing on a particular FBI transform and an
FBI-side harmonic oscillator $P_0$ adapted to the symbol $q$. Combining Theorems 2.10, 3.8, and 4.8 in~\cite{AlVi} gives a result similar to Proposition 4.1,
which reads that
$$
e^{\frac{t^{2k_0 + 1}}{C_0}P_0}e^{-tQ} \in {\cal L}(H_{\Phi_0}(\comp^n),H_{\Phi_0}(\comp^n)),\quad t\in [0,t_0],\quad 0 < t_0 \ll 1,
$$
and that $P_0$ may be replaced by the usual harmonic oscillator on the real side. Examining the proof of Theorem 2.10 shows that the operator norm of this
composition is ${\cal O}(1)$, for $t\geq 0$ sufficiently small.

\bigskip
\noindent
\color{black}
It remains for us to finish the proof of Theorem 1.2, by deriving the Gelfand-Shilov bounds (\ref{eq1.10}) on the heat semigroup $e^{-tq^w}$. When doing so,
we shall combine Proposition 4.1 with some arguments of~\cite{LMPSXu}. Let us write $P = D_x^2 + x^2$ on $L^2(\real^n)$, and
let us recall that the spectrum of $P$ is given by the eigenvalues
$\lambda_{\alpha} = 2\abs{\alpha} + n$, $\alpha \in \nat^n$, and the corresponding eigenfunctions are the Hermite functions,
\begeq
\label{eq4.11}
\psi_{\alpha}(x) = H_{\alpha}(x) e^{-x^2/2},\quad \alpha \in \nat^n,
\endeq
forming an orthonormal basis of $L^2(\real^n)$. Here the Hermite polynomials $H_{\alpha}(x)$ satisfy
$$
H_{\alpha}(x) = \prod_{j=1}^n H_{\alpha_j}(x_j).
$$
Let $u \in L^2(\real^n)$. Setting $u(t) = e^{-tq^w} u$, $t\geq 0$, we can write
\begeq
\label{eq4.11.1}
u(t) = \sum_{\alpha} a_{\alpha}(t) \psi_{\alpha}, \quad a_{\alpha}(t) = (u(t),\psi_{\alpha})_{L^2}.
\endeq
An application of Proposition 4.1 shows that
$$
\sum_{\alpha} \abs{a_{\alpha}(t)}^2 e^{\frac{2t^{2k_0 + 1}}{C_0}\lambda_{\alpha}} \leq {\cal O}(1)\norm{u}_{L^2({\bf R}^n)}^2,
$$
for all $t\in [0,t_0]$, so that for all $\alpha \in \nat^n$, we have
\begeq
\label{eq4.12}
\abs{a_{\alpha}(t)} \leq {\cal O}(1) e^{-\frac{2t^{2k_0 +1}\abs{\alpha}}{C_0}}\norm{u}_{L^2({\bf R}^n)}.
\endeq
Using~\cite[Lemma A.1]{LMPSXu}, we observe next that there exists a constant $C > 0$ such that for all $\varepsilon \in (0,1/2)$
and all $\mu,\nu \in \nat^n$, we have
\begeq
\label{eq4.13}
\norm{x^{\mu}\partial_x^{\nu}\psi_{\alpha}}_{L^2({\bf R}^n)} \leq C^{1 + \abs{\mu} + \abs{\nu}} \left(\mu!\right)^{1/2} \left(\nu!\right)^{1/2}
\frac{e^{\varepsilon\abs{\alpha}}}{\eps^{(\abs{\mu} + \abs{\nu})/2}}.
\endeq
Combining (\ref{eq4.11.1}), (\ref{eq4.12}), and (\ref{eq4.13}), we get
\begin{multline*}
\norm{x^{\mu}\partial_x^{\nu}u(t)}_{L^2({\bf R}^n)} \leq
\frac{C^{1 + \abs{\mu} + \abs{\nu}} \left(\mu!\right)^{1/2} \left(\nu!\right)^{1/2}}{\eps^{({\abs{\mu} + \abs{\nu})/2}}}
\sum_{\alpha} \exp\left(\eps \abs{\alpha} - \frac{2t^{2k_0+1}}{C_0}\abs{\alpha}\right)\norm{u}_{L^2({\bf R}^n)}.
\end{multline*}
Choosing
$$
\eps =  \frac{t^{2k_0+1}}{C_0},
$$
we conclude that with a new constant $C$,
\begeq
\label{eq4.13.1}
\norm{x^{\mu}\partial_x^{\nu}u(t)}_{L^2({\bf R}^n)} \leq
\frac{C^{1 + \abs{\mu} + \abs{\nu}} \left(\mu!\right)^{1/2} \left(\nu!\right)^{1/2}}{t^{\frac{(2k_0+1)}{2}(\abs{\mu} + \abs{\nu})}}
F\left(\frac{t^{2k_0 +1}}{C_0}\right)\norm{u}_{L^2({\bf R}^n)}, \quad 0 < t \leq t_0,
\endeq
for $t_0 > 0$ sufficiently small. Here
$$
F(y) = \sum_{\alpha} e^{-y\abs{\alpha}},\quad y>0.
$$
When estimating $F(y)$ as $y\rightarrow 0$, we notice that
\begeq
\label{eq4.13.2}
F(y) = \frac{1}{(n-1)!}\sum_{m=0}^{\infty} (m+1)\ldots\,(m+n-1)e^{-ym} \leq C_n \left( 1 + \sum_{m=0}^{\infty} m^{n-1} e^{-ym}\right),
\endeq
where the constant $C_n > 0$ depends on $n$ only. Here
$$
\sum_{m=0}^{\infty} m^{n-1}e^{-ym} = \left(-\frac{d}{dy}\right)^{n-1}\sum_{m=0}^{\infty} e^{-ym} = \left(-\frac{d}{dy}\right)^{n-1} \frac{e^y}{e^y -1},
$$
and an elementary argument allows us therefore to conclude that
\begeq
\label{eq4.14}
F(y) \leq {\cal O}_n(1) y^{-n},\quad 0 < y \leq 1.
\endeq
Combining (\ref{eq4.13.1}) and (\ref{eq4.14}), we get
\begeq
\label{eq4.15}
\norm{x^{\mu}\partial_x^{\nu}u(t)}_{L^2({\bf R}^n)} \leq \frac{C^{1 + \abs{\mu} + \abs{\nu}} \left(\mu!\right)^{1/2} \left(\nu!\right)^{1/2}}
{t^{\frac{(2k_0+1)}{2}(\abs{\mu} + \abs{\nu} +2n)}} \norm{u}_{L^2({\bf R}^n)}, \quad 0 < t \leq t_0.
\endeq
To pass to the $L^{\infty}$--norms, it suffices to apply the Sobolev embedding theorem,
\begeq
\label{eq4.16}
\norm{x^{\mu}\partial_x^{\nu}u(t)}_{L^{\infty}({\bf R}^n)} \leq {\cal O}(1) \sum_{\abs{\alpha} \leq s}
\norm{D^{\alpha}\left(x^{\mu}\partial_x^{\nu}u(t)\right)}_{L^2({\bf R}^n)},
\endeq
where $s > n/2$ is an integer. The estimate (\ref{eq1.10}) follows from (\ref{eq4.15}) and (\ref{eq4.16}), and this
completes the proof of Theorem \ref{theo_main2}.


\begin{thebibliography}{40}

\bibitem{AlVi} A. Aleman and J. Viola, {\it On weak and strong solution o\-pe\-ra\-to\-r\-s for evo\-lu\-tion equa\-ti\-ons coming f\-rom quad\-ra\-tic operators},
J. Spectral Theory, to appear.

\bibitem{CaGrHiSj} E. Caliceti, S. Graffi, M. Hitrik, and J. Sj\"ostrand, {\it Quadratic ${\cal PT}$--symmetric operators with real spectrum and similarity
to self-adjoint operators}, J. Phys. A: Math. Theor., {\bf 45} (2012), 444007.

\bibitem{DiSj} M. Dimassi and J. Sj\"ostrand, {\it Spectral asymptotics in the semi-classical limit}, Cambridge University Press, 1999.

\bibitem{GM} S. Gadat and L. Miclo, {\it Spectral decompositions and $L^2$-operator norms of toy hypocoercive semi-groups},
Kinet. Relat. Models, {\bf 6} (2013), 317-–372.

\bibitem{Hel} B. Helffer, {\it Th\'eorie spectrale pour des op\'erateurs globalement elliptiques}, Ast\'erisque, 112. Soci\'et\'e Math\'ematique de France,
Paris, 1984.

\bibitem{HelfferNier} B. Helffer and F. Nier,
{\it Hypoelliptic estimates and spectral theory for Fokker-Planck operators and Witten laplacians}, SLN 1862, Springer Verlag, 2005.

\bibitem{HeSjSt} F. H\'erau, J. Sj\"ostrand, and C. Stolk, {\it Semiclassical analysis for the Kramers-Fokker-Planck equation},
Comm. PDE, {\bf 30} (2005), 689--760.

\bibitem{HPS09} M. Hitrik and K. Pravda-Starov, {\it Spectra and semigroup smoothing for non-elliptic quadratic operators}, Math. Ann. {\bf 344} (2009), 801--846.


\bibitem{HPS13} M. Hitrik and K. Pravda-Starov, {\it Eigenvalues and subelliptic estimates for non-selfadjoint semiclassical operators with double characteristics},
Annales Inst. Fourier, {\bf 63} (2013), 985--1032.

\bibitem{HPSV1} M. Hitrik, K. Pravda-Starov, and J. Viola, {\it Short-time asymptotics of the regularizing effect for semigroups generated by quadratic operators},
preprint, 2015.

\bibitem{HiSj15} M. Hitrik and J. Sj\"ostrand, {\it Two minicourses on analytic microlocal analysis}, "Algebraic and Analytic Microlocal Analysis", 
M. Hitrik, D. Tamarkin, B. Tsygan, and S. Zelditch, eds. Springer, to appear. 

\bibitem{HiSjVi} M. Hitrik, J. Sj\"ostrand, and J. Viola, {\it Resolvent estimates for elliptic quadratic differential operators},
Anal. PDE {\bf 6} (2013), 181–196.

\bibitem{Ho95}
L. H\"{o}rmander, {\it Symplectic classification of quadratic forms, and general Mehler formulas}, Math. Z., {\bf 219} (1995), 413--449.

\bibitem{LMPSXu} N. Lerner, Y. Morimoto, K. Pravda-Starov, and C.-J. Xu, {\it Gelfand-Shilov and Gevrey smoothing effect for the spatially inhomogeneous
non-cutoff Kac equation}, J. Funct. Anal. {\bf 269} (2015), 459--535.

\bibitem{Nier} F. Nier, {\it Boundary conditions and subelliptic estimates for geometric Kra\-mers-Fok\-ker-Planck ope\-rators on manifolds with boundaries},
preprint, 2014, {\sf http://arxiv.org/abs/1309.5070}.

\bibitem{OPaPS} M. Ottobre, G. Pavliotis, and K. Pravda-Starov, {\it Exponential return to equilibrium for hypoelliptic quadratic systems},
J. Funct. Anal. {\bf 262} (2012), 4000--4039.

\bibitem{Pazy} A. Pazy, {\it Semigroups of linear operators and applications to partial differential equations}, Springer-Verlag, New York, 1983.

\bibitem{PS} K. Pravda-Starov, {\it Subelliptic estimates for quadratic differential operators}, American J. Math. {\bf 133} (2011), 39--89.


\bibitem{Sj82} J. Sj\"ostrand, {\it Singularit\'es analytiques microlocales}, Ast\'erisque, {\bf 95} (1982), 1--166,
Soc. Math. France, Paris.

\bibitem{Sj83} J. Sj\"ostrand, {\it Analytic wavefront sets and operators with multiple characteristics}, Hokkaido
Math. Journal, {\bf 12}, no. 3, part 2, 392-433 (1983).

\bibitem{Sj90} J. Sj\"ostrand, {\it Geometric bounds on the density of resonances for semiclassical problems}, Duke Math. J. {\bf 60} (1990), 1--57.

\bibitem{Sj95} J. Sj\"ostrand, {\it Function spaces associated to global I-Lagrangian manifolds}, Structure of solutions of differential
equations, Katata/Kyoto, 1995, World Sci. Publ., River Edge, NJ (1996).

\bibitem{Sj10} J. Sj\"ostrand, {\it Resolvent estimates for non-selfadjoint operators via semigroups},
Around the research of Vladimir Maz'ya. III, 359–384, Int. Math. Ser. (N. Y.), 13, Springer, New York, 2010.

\bibitem{Viola} J. Viola, {\it Spectral projections and resolvent bounds for partially elliptic quadratic differential operators},
J. Pseudo-Differ. Oper. Appl. {\bf 4} (2013), 145-–221.

\bibitem{Zworski} M. Zworski, {\it Semiclassical analysis}, American Math. Society, 2012.



\end{thebibliography}
\end{document}